\documentclass[11pt]{article}
\usepackage[left=1in,top=1in,right=1in,bottom=1in,head=.1in,nofoot]{geometry}

\usepackage{setspace,url,bm,amsmath} 

\usepackage{titlesec} 
\titlelabel{\thetitle.\quad} 

\usepackage{graphicx} 
\usepackage{bbm}
\usepackage{latexsym}
\usepackage{caption}
\usepackage[margin=20pt]{subcaption}

\usepackage[table]{xcolor}

\newcommand{\GG}[1]{}

\usepackage{amsthm}
\usepackage{amssymb}
\usepackage{amsmath}
\usepackage{color}

\usepackage{comment}
\theoremstyle{definition}

\newtheorem*{theorem*}{Theorem}
\newtheorem{theorem}{Theorem}
\newtheorem*{rmk*}{Remark}
\newtheorem{proposition}{Proposition}
\newtheorem{lemma}{Lemma}

\newtheorem{corollary}{Corollary}
\newtheorem*{corollary*}{Corollary}
\newtheorem{condition}{Condition}

\usepackage{natbib} 
\bibpunct{(}{)}{;}{a}{}{,} 

\usepackage{etoolbox} 
\apptocmd{\sloppy}{\hbadness 10000\relax}{}{} 

\usepackage{multibib}
\newcites{sec}{REFERENCES}

\usepackage{color}
\usepackage{listings}
\usepackage{hyperref}

\def\obs{\text{obs}}

\def\Var{\text{Var}}
\def\Cov{\text{Cov}}

\def\converged{\stackrel{d}{\longrightarrow}}

\def\Tau{{T}}

\def\apprsim{\overset{.}{\sim}}
\def\QR{\text{QR}}

\def\colpf{\color{black}}

\RequirePackage[normalem]{ulem}

\allowdisplaybreaks
\begin{document}
\doublespacing
\title{\bf Asymptotic Theory of Rerandomization in Treatment-Control Experiments
}
\author{Xinran Li, Peng Ding and Donald B. Rubin
\footnote{Xinran Li is Doctoral Candidate, Department of Statistics, Harvard University,  Cambridge, MA 02138 (E-mail: \url{xinranli@fas.harvard.edu}). Peng Ding is Assistant Professor, Department of Statistics, University of California, Berkeley, CA 94720 (E-mail: \url{pengdingpku@berkeley.edu}).
Donald B. Rubin is Professor, Department of Statistics, Harvard University, Cambridge, MA 02138 (E-mail: \url{rubin@stat.harvard.edu}).
}
}
\date{}
\maketitle
\begin{abstract}
Although complete randomization ensures covariate balance on average, the chance for observing significant differences between treatment and control covariate distributions increases with many covariates. Rerandomization discards randomizations that do not satisfy a predetermined covariate balance criterion, generally resulting in better covariate balance and more precise estimates of causal effects. Previous theory has derived finite sample theory for rerandomization under the assumptions of equal treatment group sizes, Gaussian covariate and outcome distributions, or additive causal effects, but not for the general sampling distribution of the difference-in-means estimator for the average causal effect. To supplement existing results, we develop asymptotic theory for rerandomization without these assumptions, which reveals a non-Gaussian asymptotic distribution for this estimator, specifically a linear combination of a Gaussian random variable and a truncated Gaussian random variable. This  distribution follows because rerandomization affects only the projection of potential outcomes onto the covariate space but does not affect the corresponding orthogonal residuals. We also demonstrate that, compared to complete randomization, rerandomization reduces the asymptotic sampling variances and quantile ranges of the difference-in-means estimator. Moreover, our work allows the construction of accurate large-sample confidence intervals for the average causal effect, thereby revealing further advantages of rerandomization over complete randomization. 
\end{abstract}

\newpage 
\section{Introduction}
Ever since \citet*{fisher1925statistical, fisher1926, fisher:1935}'s seminal work, randomized experiments have become the ``gold standard" for drawing causal inferences. Complete randomization balances the covariate distributions between treatment groups in expectation, thereby ensuring the existence of unbiased estimators of the average causal effect.  Covariate imbalance, however, often occurs in specific randomized experiments, as recognized by \citet{fisher1926} and later researchers \citep[e.g.,][]{student:1938, Greevy01042004, hansen2008covariate, keele2009adjusting, Bruhn:2009, kriegernearlyrandom2016, ATHEY201773}. The standard approach advocated by \citet{fisher:1935}, stratification or blocking, ensures balance with a few discrete covariates; see \citet{cochrancox} and \citet{imbens2015causal} for detailed discussions.

When a randomized allocation is unbalanced, it is reasonable to discard that allocation and re-draw another one until a certain pre-determined covariate balance criterion is satisfied. This is rerandomization, an experimental design hinted by R. A. Fisher \citep[cf.][page 88]{savage:1962} and \citet{cox:1982, cox:2009}, and formally proposed by \citet{rubin2008comment} and \citet{morgan2012rerandomization}. Note that rerandomization is conceptually the same as restricted or constrained randomization \citep[e.g.,][]{yates1948comment, grundy1950restricted, youden1972randomization, bailey1983restricted}. For more historical discussion, see \citet*[][page 45]{fienberg:1980}, \citet{speed1992introduction}, \citet[][page 57]{lehmann2011fisher}, and \citet{morgan2012rerandomization}.

\citet{morgan2012rerandomization} showed that the difference-in-means estimator is generally unbiased for the average causal effect under rerandomization with equal-sized treatment groups, and obtained the sampling variance of this estimator under additional assumptions of Gaussian covariate and outcome distributions and additive causal effects. When rerandomization is applied when these assumptions do not hold, statistical inference becomes more challenging, because the Gaussian distributional theory that is justified by the central limit theorem under complete randomization \citep[cf.][]{hajek1960limiting, lin2013} no longer generally holds. Some applied researchers believe that ``the only analysis that we can be completely confident in is a permutation test or rerandomization test" \citep{Bruhn:2009}. However, randomization-based tests require sharp null hypotheses that all individual causal effects are known from observed values.

Analogous to the repeated sampling properties for complete randomization \citep{Neyman:1923, imbens2015causal}, 
we evaluate the sampling properties of the difference-in-means estimator when rerandomization is used, where all potential outcomes and covariates are regarded as fixed quantities and all randomness arises solely from the random treatment assignments. The geometry of rerandomization reveals non-Gaussian asymptotic distributions, which serve as the foundation for constructing large-sample confidence intervals for average causal effects. 
Furthermore, we compare the lengths of quantile ranges of the asymptotic distributions of the difference-in-means estimator under rerandomization and complete randomization, extending \citet{morgan2012rerandomization, morgan2015rerandomization}'s comparison of their sampling variances.

\section{Framework, Notation, and Basic Results}\label{sec:notation}

\subsection{Covariate imbalance and rerandomization}

Inferring the causal effect of some binary treatment on an outcome $Y$ is of central interest in many studies.
We consider an experiment with $n$ units, with $n_1$ assigned to treatment and $n_0$ assigned to control, $n=n_1 + n_0$, indexed by $i=1,\ldots,n$. Before conducting the experiment, we collect $K$ covariates $\bm{X}_i=(X_{1i}, X_{2i}, \ldots, X_{Ki})$ for each unit, 
which can possibly include transformations of basic covariates and their interactions. 
Let $Z_i$ be the indicator variable for unit $i$ assigned to treatment ($Z_i=1$ if active treatment level; $Z_i=0$ if the control level), and $\bm{Z}=(Z_1,Z_2, \ldots,Z_n)'$ be the treatment assignment column vector. In a completely randomized experiment (CRE), the distribution of $\bm{Z}$ is such that each value, $\bm{z}=(z_1,\ldots,z_n)'$, of $\bm{Z}$ has probability
$
n_1!n_0!/n!,
$
where $\sum_{i=1}^n z_i = n_1$ and $\sum_{i=1}^n (1-z_i) = n_0$, which does not depend on the values of any observed or unobserved covariates. 
The difference-in-means vector of the covariates between treatment and control groups is
$$
\hat{\bm{\tau}}_{\bm{X}}  = \frac{1}{n_1}\sum_{i=1}^n Z_i \bm{X}_i- \frac{1}{n_0}\sum_{i=1}^n (1-Z_i) \bm{X}_i.
$$
Although on average $\hat{\bm{\tau}}_{\bm{X}}$ has mean zero over all $\binom{n}{n_1}$ randomizations, for any realized value of $\bm{Z}$, imbalancedness in covariate distributions between treatment groups often occurs. As pointed out by \citet{morgan2012rerandomization}, with $10$ independent covariates and significance level $5\%$, the probability of a significant difference for at least one covariate is $40\%$.

When significant covariate imbalance arises in a drawn allocation, it is reasonable to discard the unlucky allocation and draw another treatment assignment vector until some a priori covariate balance criterion is satisfied. This is rerandomization, an intuitive experimental design tool apparently personally advocated by R. A. Fisher \citep[see the discussion by][]{rubin2008comment} and formally discussed by \citet{morgan2012rerandomization}.

In general, rerandomization entails the following steps:
\begin{itemize}
\item[(1)] collect covariate data;
\item[(2)] specify a balance criterion to determine whether a randomization is acceptable or not;
\item[(3)] randomize the units to treatment and control groups;
\item[(4)] if the balance criterion is satisfied, proceed to Step (5); otherwise, return to Step (3);
\item[(5)] conduct the experiment using the final randomization obtained in Step (4);
\item[(6)] analyze the data taking into account the rerandomization used in Steps (2)--(4).
\end{itemize}

Although the balance criterion in Step (2) can be general, \citet{morgan2012rerandomization} suggested using the Mahalanobis distance between covariate means in treatment and control groups, and \citet{morgan2015rerandomization} suggested considering tiers of covariates according to their presumed importance in predicting the outcomes in this experiment. We will discuss these two types of rerandomization in detail, and apposite statistical inference after these rerandomizations as implied by Step (6). We then extend the theory to rerandomizations under more general covariate balance criteria in Section \ref{sec:ReG}.

\subsection{Potential outcomes and definitions of finite population quantities}

We use the potential outcomes framework \citep[sometimes called the Rubin Causal Model; see ][]{holland1986statistics, imbens2015causal} to define causal effects,
and let $Y_i(1)$ and $Y_i(0)$ denote the potential outcomes of unit $i$ under active treatment and control, respectively. 
On the difference scale, 
the individual causal effect for unit $i$ is $\tau_i=Y_i(1)-Y_i(0)$, and the average causal effect in the finite population of $n$ units is $\tau_Y = \sum_{i=1}^n \tau_i/n$. 
Let $\bar{Y}(z) = \sum_{i=1}^{n}Y_i(z)/n$ be the finite population average of potential outcomes under treatment arm $z$, $\bar{\bm{X}}$ the finite population average of covariates,
$S_{Y(z)}^2$ the finite population variance (with divisor $n-1$) of the potential outcomes under treatment arm $z$, $S_{Y(z),\bm{X}}=S_{\bm{X},Y(z)}'$ the finite population covariance between potential outcomes and covariates, and $\bm{S}_{\bm{X}}^2$  the finite population covariance matrix of covariates. 
For simplicity, we avoid notation for these quantities' dependence on $n$. Notice that these quantities are fixed, and are not dependent on the randomization or rerandomization scheme.

\subsection{Repeated sampling inference in a CRE}
\label{sec::CRE-neyman}

The observed outcome for unit $i$ is $Y_i = Z_iY_i(1)+(1-Z_i)Y_i(0)$, a function of treatment assignment and potential outcomes. In a CRE, \citet{Neyman:1923} showed that, for estimating $\tau_Y$, the difference-in-means estimator
$$
\hat{\tau}_Y = \frac{1}{n_1}\sum_{i=1}^n Z_i Y_i - \frac{1}{n_0}\sum_{i=1}^n (1-Z_i) Y_i
$$
is unbiased (the expectation of $\hat{\tau}_Y$ over all randomizations is $\tau_Y$), and obtained its sampling variance over all randomizations for constructing a large-sample confidence interval for $\tau_Y$. However, \citet{Neyman:1923}'s interval is not accurate if rerandomization is used, except in an asymptotic conservative sense.

Let $r_1 = n_1/n$ and $r_0 = n_0/n$ be the proportions of units receiving treatment and control. 
According to
the finite population central limit theorem \citep[][]{hajek1960limiting}, under some regularity conditions, the large $n$ sampling distribution, over all randomizations, of
$\sqrt{n}(\hat{\tau}_Y-{\tau}_Y, \hat{\bm{\tau}}_{\bm{X}}')$ is Gaussian with mean zero and covariance matrix $\bm{V}$, where
\begin{align*}
\bm{V}
= 
\begin{pmatrix}
V_{\tau\tau} & \bm{V}_{\tau \bm{x}}\\
\bm{V}_{\bm{x}\tau} & \bm{V}_{\bm{xx}}
\end{pmatrix}
 = 
\begin{pmatrix}
{r_1^{-1}}S_{Y(1)}^2 + {r_0^{-1}}S_{Y(0)}^2 - S_{\tau}^2 & {r_1^{-1}}\bm{S}_{Y(1),\bm{X}}+{r_0^{-1}}\bm{S}_{Y(0),\bm{X}}\\
{r_1^{-1}}\bm{S}_{\bm{X}, Y(1)}+{r_0^{-1}}\bm{S}_{\bm{X},Y(0)} & (r_1r_0)^{-1}\bm{S}_{\bm{X}}^2
\end{pmatrix}.
\end{align*}
Note again that we are conducting randomization-based inference, where all the covariates and potential outcomes are fixed numbers, and randomness comes solely from the treatment assignment. We embed $n$ units into an infinite sequence of finite populations with increasing sizes, and a sufficient condition for the asymptotic Gaussianity of $\sqrt{n}(\hat{\tau}_Y-{\tau}_Y, \hat{\bm{\tau}}_{\bm{X}}')$ is
as follows \citep{fcltxlpd2016}. 
\begin{condition}\label{cond:fp}
As $n\rightarrow\infty$, for $z=0,1$, 
\begin{itemize}
\item[(i)] $r_z$, the proportion of units under treatment arm $z$, has positive limits, 
\item[(ii)] the finite population variances and covariances $S^2_{Y(z)}, S^2_{\tau}, \bm{S}^2_{\bm{X}}$ and $S_{\bm{X},Y(z)}$ have limiting values, 
\item[(iii)] $\max_{1\leq i\leq n} |Y_i(z) - \bar{Y}(z)|^2/n \rightarrow 0$ and 
$
\max_{1\leq i\leq n}\|\bm{X}_i - \bar{\bm{X}}\|_2^2/n \rightarrow 0. 
$ 
\end{itemize}
\end{condition}
We introduce the notation $\apprsim$ for two sequences of random vectors converging weakly to the same distribution. Therefore, 
under CRE and Condition \ref{cond:fp}, 
$
\sqrt{n}(\hat{\tau}_Y-{\tau}_Y, \hat{\bm{\tau}}_{\bm{X}}')
\apprsim
(A, \bm{B}'),
$
where $(A, \bm{B}')$ is a random vector from $\mathcal{N}(\bm{0},\bm{V})$.

\section{Rerandomization using the Mahalanobis distance}
\label{sec::ReM}

\subsection{Mahalanobis distance}
 
The Mahalanobis distance between the covariate means in treatment and control groups is 
$$
\hat{\bm{\tau}}_{\bm{X}}' \{ \Var(\hat{\bm{\tau}}_{\bm{X}})\}^{-1}\hat{\bm{\tau}}_{\bm{X}} = \left( \sqrt{n}\hat{\bm{\tau}}_{\bm{X}}\right)' \bm{V}_{\bm{xx}}^{-1} 
\left( \sqrt{n}\hat{\bm{\tau}}_{\bm{X}} \right),
$$ 
recalling that $  \bm{V}_{\bm{xx}} =  (r_1r_0)^{-1}\bm{S}_{\bm{X}}^2  $ is a fixed and known $K\times K$ matrix in our finite population setting.
A rerandomization scheme proposed by \citet{morgan2012rerandomization} accepts only those randomizations with the Mahalanobis distance less than or equal to $a$, a pre-specified threshold. 
Let
$$
\mathcal{M} = \{\bm{\bm{\mu}}: \bm{\bm{\mu}}' \bm{V}_{\bm{xx}}^{-1}\bm{\bm{\mu}} \leq a \}
$$
denote the acceptance region for $\sqrt{n}\hat{\bm{\tau}}_{\bm{X}}$; that is, a treatment assignment vector $\bm{Z}$ is accepted if and only if the corresponding $\sqrt{n}\hat{\bm{\tau}}_{\bm{X}} \in \mathcal{M}$. Below we use ReM to denote rerandomization using this criterion.

Several practical issues with ReM are worth mentioning. First, if we include transformations and interactions of $\bm{X}$, then ReM can incorporate a wide class of rerandomization schemes. Second, for small sample sizes, it can be that there does not exist any randomization satisfying the balance criterion. However, according to the finite population central limit theorem, the acceptance probability of a randomization is asymptotically $p_a=P(\chi^2_K\leq a)$. Therefore, for relatively large sample size, there usually exist many randomizations satisfying the balance criterion with $a>0$. In practice, we would like to choose the asymptotic acceptance probability to be small, e.g., $p_a = 0.001$. However, we do not want $p_a$ to be too small, such as accepting only those assignments with the smallest Mahalanobis distance. Too small $p_a$ will result in few randomizations, making the repeated sampling inference intractable, even asymptotically, as well as the randomization tests powerless \citep{morgan2012rerandomization}. Furthermore, as illustrated by later examples, the gain from reducing $p_a$ usually decreases as $p_a$ becomes smaller.

\subsection{Multiple correlation between potential outcomes and covariates}\label{sec::R2}

We define the finite population squared multiple correlation between the potential outcome $Y(z)$ and the covariates $\bm{X}$ as $R^2(z)$ for $z=1,0$,
and the finite population squared multiple correlation between the individual causal effect and the covariates as $R^2(\tau)$. Note that $R^2(1), R^2(0)$ and $R^2(\tau)$ are quantities of the finite population, which do not depend on the randomization or rerandomization scheme. 
Similar measures also appeared in \citet{cochran1965} and \citet{rubin1976multivariate}.

We further define an $R^2$-type measure that is a function of the finite population quantities as well as the proportions of the group sizes:
$$
R^2 = \frac{S^2_{Y(1)}}{r_1 V_{\tau \tau}}R^2(1) + 
\frac{S^2_{Y(0)}}{r_0V_{\tau \tau}}R^2(0) - 
\frac{S^2_{\tau}}{V_{\tau \tau}}R^2(\tau).
$$
When the causal effect is additive, $S_{\tau}^2=0$ and $S^2_{Y(1)}=S^2_{Y(0)}$, and then
$
R^2 = R^2(1) = R^2(0)
$
reduces to the squared multiple correlation between $\bm{X}$ and $Y(1)$ or $Y(0)$.

The following proposition states that under CRE $R^2$ is the proportion of  the sampling variance of $\hat{\tau}_Y$ explained by $\hat{\bm{\tau}}_{\bm{X}}$ in linear projection.

\begin{proposition}\label{prop:R_sampling_cor}
The sampling squared multiple correlation between $\hat{\tau}_Y$ and $\hat{\bm{\tau}}_{\bm{X}}$ under CRE is $R^2$, which can be equivalently written as
\begin{align*}
R^2 & = \text{Corr}(\hat{\tau}_Y, \hat{\bm{\tau}}_{\bm{X}}) 
= \frac{r_1^{-1}S_{Y(1)\mid \bm{X}}^2+r_0^{-1}S_{Y(0)\mid \bm{X}}^2
- S_{\tau\mid \bm{X}}^2
}{{r_1^{-1}}S_{Y(1)}^2 + {r_0^{-1}}S_{Y(0)}^2 - S_{\tau}^2},
\end{align*} 
where $S_{Y(z)\mid \bm{X}}^2$ and $S_{\tau\mid \bm{X}}^2$ are the finite population variances of the linear projections of the potential outcomes and individual causal effects on covariates.
\end{proposition}

\subsection{Asymptotic sampling distribution of $\hat{\tau}_Y$ under ReM}

With rerandomization, we accept the randomizations satisfying the covariate balance criterion, and therefore the sampling distribution of $\sqrt{n}(\hat{\tau}_Y-\tau_Y)$ over rerandomizations is the same as its sampling distribution over a CRE conditional on $\sqrt{n}\hat{\bm{\tau}}_{\bm{X}}$ satisfying the covariate balance criterion. Although the following proposition holds for rerandomization with more general balance criteria, we first state it for ReM.

\begin{proposition}\label{thm:asymp_dist_maha}
Under ReM and Condition \ref{cond:fp}, 
\begin{align}\label{eq:asym_dist_gen}
\left.
\begin{pmatrix}
\sqrt{n}(\hat{\tau}_Y-\tau_Y) \\
\sqrt{n}\hat{\bm{\tau}}_{\bm{X}}
\end{pmatrix}
 \ \right| \ \sqrt{n}\hat{\bm{\tau}}_{\bm{X}}\in \mathcal{M} \ \ 
\apprsim \ \ 
\left.
\begin{pmatrix}
A\\
\bm{B}
\end{pmatrix}
\ 
\right| \ 
\bm{B} \in \mathcal{M},
\end{align}
recalling from earlier that $(A, \bm{B}')$ is a random vector following $\mathcal{N}(\bm{0},\bm{V})$.
\end{proposition}

Simply stated, $\sqrt{n}(\hat{\tau}_Y-\tau_Y)$ has two parts: the part unrelated to the covariates, which we call $\varepsilon_0$, and thus unaffected by  rerandomization, and the other part related to the covariates, which we call $L_{K,a}$, and thus affected by  rerandomization. 
Therefore, the asymptotic distribution of $\hat{\tau}_Y$ is a linear combination of two independent random variables:
$\varepsilon_0\sim \mathcal{N}(0,1)$ is a standard Gaussian random variable, and
$L_{K,a}$ is a random variable following the distribution of $D_1\mid \bm{D}'\bm{D}\leq a$, where $\bm{D}=(D_1,\ldots,D_K)'\sim \mathcal{N}(\bm{0}, \bm{I}_K)$.

\begin{theorem}\label{thm:asymp_dist_Y_maha}
Under ReM and Condition \ref{cond:fp}, 
\begin{eqnarray}\label{eq:dist_maha_Gaussian_trun_norm}
\sqrt{n}(\hat{\tau}_Y-\tau_Y) \mid  \sqrt{n}\hat{\bm{\tau}}_{\bm{X}} \in \mathcal{M}    
 \apprsim
\sqrt{V_{\tau\tau}}\left( 
\sqrt{1-R^2} \cdot \varepsilon_0 + \sqrt{R^2} \cdot L_{K,a}   \right),
\end{eqnarray} 
where
$\varepsilon_0$ is independent of $L_{K,a}$.
\end{theorem}

The coefficients of the linear combination are functions of $R^2$, which measures the association between the potential outcomes and the covariates. When $R^2=0$, the right hand side of (\ref{eq:dist_maha_Gaussian_trun_norm}) becomes a Gaussian random variable, the same as the asymptotic distribution of $\sqrt{n}(\hat{\tau}_Y-\tau_Y)$ under CRE in Section \ref{sec::CRE-neyman}; when $R^2=1$, \eqref{eq:dist_maha_Gaussian_trun_norm} reduces to $\sqrt{V_{\tau\tau}}\cdot L_{K,a}$, a random variable with bounded support $[-\sqrt{aV_{\tau\tau}},\sqrt{aV_{\tau\tau}}]$. Importantly, the definition of $R^2$ is based on linear projections but not linear models of the potential outcomes. Our asymptotic theory is based on the distribution of the randomization without imposing any modeling assumptions on the potential outcomes. Furthermore, under rerandomization, the asymptotic distribution in \eqref{eq:dist_maha_Gaussian_trun_norm} has a clear geometric interpretation as displayed in Figure \ref{fig:geo_rerand}, in which we fix $V_{\tau\tau}$ at $1$ without loss of generality, $\theta$ is the angle between $\sqrt{n}(\hat{\tau}_Y - \tau_Y)$ and its projection on the space spanned by $\sqrt{n}\hat{\bm{\tau}}_{\bm{X}}$, and then $R$ is the cosine of $\theta$.

\begin{figure}[htb]
  \centering
    \includegraphics[width=0.6\textwidth]{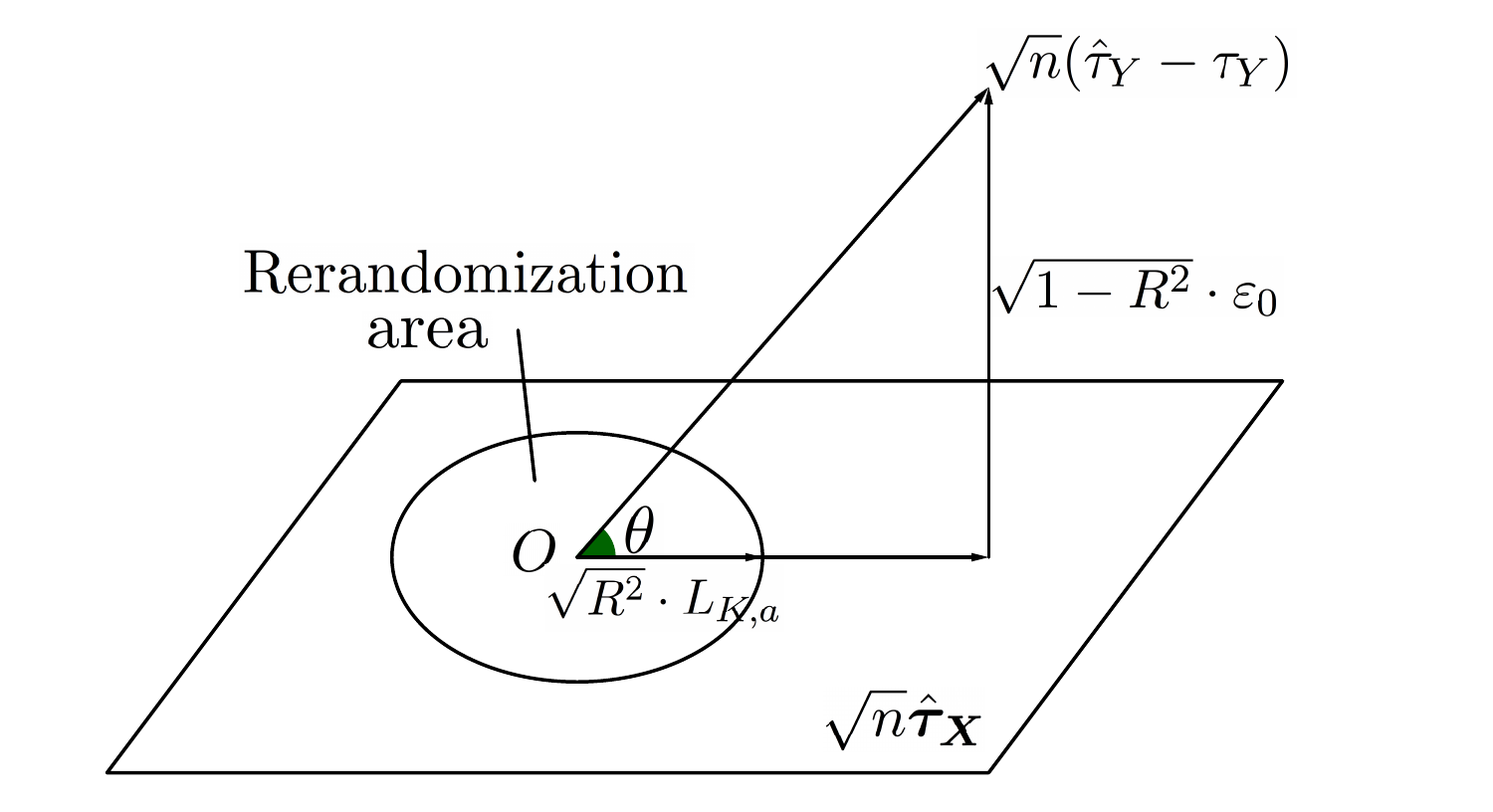}
   \caption{
   Geometry of rerandomization.
   $O$ is the origin, $\theta$ is the angle between $\sqrt{n}(\hat{\tau}_Y-\tau_Y)$ and its projection on $\sqrt{n}\hat{\bm{\tau}}_{\bm{X}}$. The ellipse around $O$ is the acceptance region under ReM. 
   $\sqrt{1-R^2}\cdot \varepsilon_0$ is the component of $\sqrt{n}(\hat{\tau}_Y-\tau_Y)$ orthogonal to the space of $\sqrt{n}\hat{\bm{\tau}}_{\bm{X}}$, and $\sqrt{R^2}\cdot L_{K,a}$ is the projection onto the space of $\sqrt{n}\hat{\bm{\tau}}_{\bm{X}}$ under ReM.
   }\label{fig:geo_rerand}
\end{figure}

\subsection{Representation and simulation of the asymptotic distribution under ReM}

The asymptotic distribution in \eqref{eq:dist_maha_Gaussian_trun_norm} involves a random variable $L_{K,a}$ that does not appear in standard statistical problems. Algebraically, $L_{K,a} \sim D_1 \mid \bm{D}'\bm{D}\leq a$ is the first coordinate of a $K$ dimensional standard Gaussian vector, subject to the constraint that the squared length of the vector does not exceed $a$. This type of truncation of Gaussian distributions is apparently unstudied except for \citet{tallis1963} 
and \citet{morgan2012rerandomization}. 
Because the standard Gaussian vector is spherically symmetric \citep{dempster1969elements, rubin1976multivariate, fang1989symmetric}, it can be written as a product of two independent random components, a $\chi_K$ random variable and a random vector uniformly distributed on the $(K-1)$ dimensional unit sphere. The truncation condition, $\bm{D}'\bm{D}\leq a$, affects only the first component $\chi_K$, leaving the second component unchanged. Basic properties of spherically symmetrical distributions allow us to represent $L_{K,a}$ using some known distributions, which allows for easy simulation of $L_{K,a}.$

Let $\chi^2_{K,a}  \sim \chi^2_K \mid \chi^2_K \leq a$ be a truncated $\chi^2$ random variable, $U_K$ the first coordinate of the uniform random vector over the $(K-1)$ dimensional unit sphere, $S$ a random sign taking $\pm 1$ with probability $1/2$, and $\beta_K \sim \text{Beta}\left(1/2, (K-1)/2\right)$ a Beta random variable degenerating to a point mass at $1$ when $K=1$.

\begin{proposition}\label{prop::curious}
$L_{K,a}$ can be represented as
\begin{eqnarray}\label{eq::curious}
L_{K,a} \sim D_1 \mid \bm{D}'\bm{D}\leq a
\sim \chi_{K,a} U_K
\sim \chi_{K,a} S \sqrt{\beta_K},   
\end{eqnarray}
where 
$(\chi_{K,a},  U_K)$ are mutually independent, and
$(\chi_{K,a} , S,   \beta_K )$ are mutually independent. 
$L_{K,a}$ is symmetric and unimodal around zero, with variance
$\Var(L_{K,a})=v_{K,a}=P(\chi^2_{K+2}\leq a) / P(\chi^2_K \leq a)<1$.
\end{proposition}

Because both $\varepsilon_0$ and $L_{K,a}$ are symmetric and both are unimodal at zero, their linear combination is also symmetric and unimodal at zero according to \citet{wintner1936class}'s Theorem. The same is true for the asymptotic distribution of $\sqrt{n}(\hat{\tau}_Y - \tau_Y)$ in (\ref{eq:dist_maha_Gaussian_trun_norm}).
The representation in \eqref{eq::curious} allows for easy simulation of $L_{K,a}$, as well as the asymptotic distribution of $\sqrt{n}(\hat{\tau}_Y - \tau_Y)$ in \eqref{eq:dist_maha_Gaussian_trun_norm}, which is relevant for statistical inference discussed later.

Without loss of generality, we fix $V_{\tau\tau}$ at $1$, and consider the distribution of
\begin{align}\label{eq:std_dist}
Q = 
\sqrt{1-R^2}\cdot \varepsilon_0 + \sqrt{R^2}\cdot L_{K,a},
\end{align}
which depends on $R^2$, the dimension of the covariates $K$, and the asymptotic acceptance probability of rerandomization $p_a =  P(\chi^2_K \leq a)$. 
We simulate values of $Q$ using independent and identically distributed (i.i.d) draws from (\ref{eq:std_dist}).
First, we fix $K=10$ and $p_a=0.001$. Figure  \ref{fig:std_dist_R2_K10_pa0001} shows the probability densities of $Q$ with different values of $R^2$, which approaches to that of $L_{K,a}$ as $R^2$ increases. Because $\varepsilon_0$ is more diffusely distributed than the truncated variable $L_{K,a}$, the probability density of $Q$ will concentrate more around $0$ with increasing $R^2$, as shown in Figure  \ref{fig:std_dist_R2_K10_pa0001}. 

Second,
we fix $K=3$ and $R^2=0.6$. Figure \ref{fig:std_dist_p_a} shows the probability densities of $Q$ with different values of asymptotic  acceptance probability $p_a$; the CRE corresponds to $p_a=1$. With smaller $p_a$, the distribution of $Q$ becomes more concentrated around $0$. Asymptotically, using smaller acceptance probabilities in ReM gives us more precise estimators for the average causal effect. 
However, when $R^2<1,$ which is usually the case in practice, the gain of ReM by decreasing the threshold $a$ becomes less as $a$ becomes smaller. For example, the density of $Q$ with $p_a = 0.0001$ is almost the same as the one with $p_a=0.001$ in Figure \ref{fig:std_dist_p_a}, and the percentage reduction in variance of $Q$ achieved by decreasing $p_a$ from $0.001$ to $0.0001$ is only $5.7\%.$

\begin{figure}
\centering
\begin{subfigure}{.5\textwidth}
  \centering
  \includegraphics[width=0.8\linewidth]{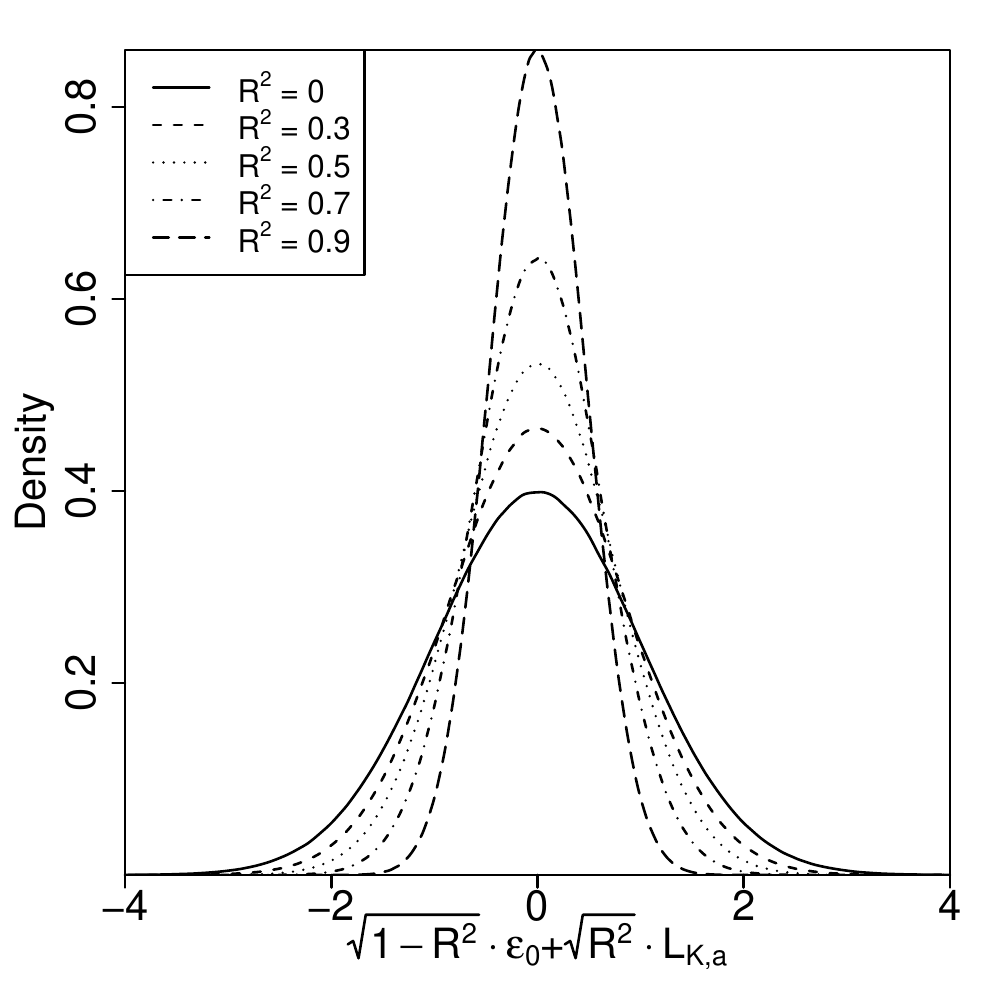}
      \caption{$K=10$ and $p_a=0.001$} \label{fig:std_dist_R2_K10_pa0001}
\end{subfigure}%
\begin{subfigure}{.5\textwidth}
  \centering
  \includegraphics[width=0.8\linewidth]{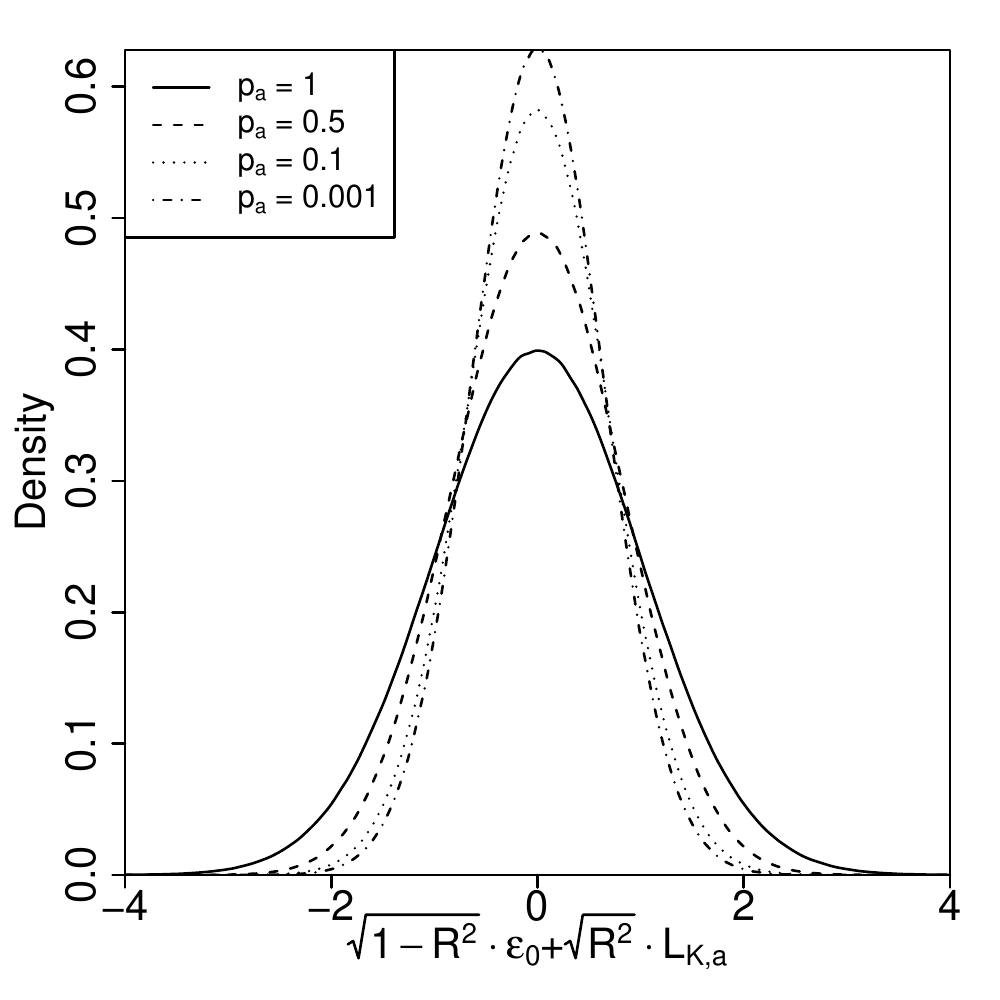}
      \caption{\centering $R^2=0.6$ and $K=3$}
    \label{fig:std_dist_p_a}
\end{subfigure}
\caption{Asymptotic distribution under ReM with $V_{\tau\tau}$ fixed at $1$}
\label{fig:std_dist_maha}
\end{figure}

\subsection{Asymptotic unbiasedness, sampling variance and quantile ranges}

Theorem \ref{thm:asymp_dist_Y_maha} characterizes the asymptotic behavior of $\hat{\tau}_Y$ over 
ReM,
which immediately implies the following conclusions as extensions of \citet{morgan2012rerandomization}.

First, the asymptotic distribution in (\ref{eq:dist_maha_Gaussian_trun_norm}) is symmetric around $0$, implying that $\hat{\tau}_Y$ is asymptotically unbiased for ${\tau}_Y$. 
Let $\mathbb{E}_{\text{a}}(\cdot)$ and $\Var_{\text{a}}(\cdot)$ denote the expectation and covariance matrix (or variance for scalar cases) of the asymptotic sampling distribution of a sequence of random vectors.
\begin{corollary}\label{thm:unbiased_maha}
Under ReM and Condition \ref{cond:fp}, 
$
\mathbb{E}_{\text{a}}\left\{ \sqrt{n}(\hat{\tau}_Y -\tau_Y) 
\mid \sqrt{n}\hat{\bm{\tau}}_{\bm{X}} \in \mathcal{M}
\right\}=0.
$
\end{corollary}

\citet{morgan2012rerandomization} gave a counter-example showing that, in an experiment with unequal treatment group sizes, $\hat{\tau}_Y$ can be biased for $\tau_Y$ over ReM. As conjectured by \citet{morgan2015rerandomization}, our result suggests that the bias is often small with large samples.
Corollary \ref{thm:unbiased_maha} extends \citet[][Theorem 2.1]{morgan2012rerandomization} and ensures the asymptotic unbiasedness of $\hat{\tau}_Y$ for experiments with any ratio of group sizes. 
Corollary \ref{thm:unbiased_maha} also implies that any covariate asymptotically has the same means under treatment and control.

Furthermore, from Proposition \ref{prop::curious} and Theorem \ref{thm:asymp_dist_Y_maha}, we can calculate the asymptotic sampling variances of  $\hat{\bm{\tau}}_{\bm{X}}$ and $\hat{\tau}_Y$, and the percentage reductions in asymptotic sampling variances (PRIASV) under ReM compared to CRE. Recalling that $v_{K,a} = P(\chi^2_{K+2}\leq a) / P(\chi^2_K \leq a)$, we summarize the results below.

\begin{corollary}
\label{coro::ReM-variance}
Under ReM and Condition \ref{cond:fp}, the asymptotic sampling covariance of $\hat{\bm{\tau}}_{\bm{X}}$ is 
$$
\Var_{\text{a}}\left(
\sqrt{n}\hat{\bm{\tau}}_{\bm{X}}
\mid \sqrt{n}\hat{\bm{\tau}}_{\bm{X}} \in \mathcal{M}
\right)
=v_{K,a} \bm{V}_{\bm{xx}},
$$
and the PRIASV of any component of $\hat{\bm{\tau}}_{\bm{X}}$ is
$
1-v_{K,a}.
$
The asymptotic sampling variance of $\hat{\tau}_Y$ is
\begin{align}\label{eq:symp_var_maha}
\Var_{\text{a}}\left\{
\sqrt{n}(\hat{\tau}_Y-\tau_Y)
\mid \sqrt{n}\hat{\bm{\tau}}_{\bm{X}} \in \mathcal{M}
\right\} 
= V_{\tau\tau}\left\{ 1- (1-v_{K,a})R^2
\right\} , 
\end{align}
and the PRIASV of $\hat{\tau}_Y$ is
$
(1-v_{K,a})R^2.
$
\end{corollary}

Note that the asymptotic sampling covariance and sampling variance of $\hat{\bm{\tau}}_{\bm{X}}$ and $\hat{\tau}_Y$ are actually the limits of $v_{K,a} \bm{V}_{\bm{xx}}$ and $V_{\tau\tau}\{ 1- (1-v_{K,a})R^2
\}$ in the sequence of finite populations. However, for descriptive convenience, we omit these limit signs when discussing the expectation and covariance of asymptotic sampling distributions. 
When $a$ is close to $0$, or equivalently the asymptotic acceptance probability is small, the asymptotic sampling variance on the right hand side of (5) reduces to $V_{\tau\tau} (1-R^2)$, which is identical to the asymptotic sampling variance of the regression adjusted estimator under CRE discussed in \citet{lin2013} as an extension of \citet{fisher1925statistical,fisher:1935}. Therefore, rerandomization does covariate adjustment in the design stage, and regression does covariate adjustment in the analysis stage. \citet{cox:2009} and \citet{morgan2012rerandomization} discussed related issues.

When the causal effect is additive, $R^2$ is equal to the finite population squared multiple correlation between $\bm{X}$ and $Y(0)$.
Therefore, Corollary \ref{coro::ReM-variance} is an asymptotic version of Theorem 3.2 in \citet{morgan2012rerandomization}.

Under ReM, in addition to the sampling variance reduction result concerning $\hat{\tau}_Y$ in Corollary \ref{coro::ReM-variance}, we consider the reduction in the length of the $(1-\alpha)$ quantile range of $\hat{\tau}_Y$ compared to that under CRE. We choose the length of the $(1-\alpha)$ quantile range, because of its connection to constructing confidence intervals as discussed shortly.

Let $z_\xi$ be the $\xi$th quantile of a standard Gaussian distribution. Let $\nu_{\xi}(R^2, p_a,K)$ be the  $\xi$th quantile of the distribution of $Q$ in (\ref{eq:std_dist}).
Note that $\nu_{\xi}(0,p_a,K)=z_{\xi}$. Because $p_a$ and $K$ are usually known by design, we write $\nu_{\xi}(R^2, p_a,K)$ as $\nu_{\xi}(R^2)$ for notational simplicity. 
Under ReM,  the $(1-\alpha)$ quantile range of the asymptotic distribution of $\sqrt{n}({\hat{\tau}_Y-\tau_Y})$ is  
\begin{align}\label{eq:asym_qr_maha}
\QR_{\alpha}(V_{\tau\tau}, R^2) = \left[   \nu_{\alpha/2}(R^2)\sqrt{V_{\tau\tau}},\ \   \nu_{1-\alpha/2}(R^2)\sqrt{V_{\tau\tau}}
\right],
\end{align}
and the corresponding quantile range under CRE is 
\begin{align}\label{eq:asym_qr_cre}
\QR_{\alpha}(V_{\tau\tau}, 0) = \left[z_{\alpha/2}\sqrt{V_{\tau\tau}}, \ \  z_{1-\alpha/2}\sqrt{V_{\tau\tau}}
\right].
\end{align}

\begin{theorem}\label{thm:shorter_ci_maha}
Under Condition \ref{cond:fp},
the length of the $(1-\alpha)$ quantile range of the asymptotic sampling distribution of
$\sqrt{n}(\hat{\tau}_Y-\tau_Y)$ under ReM is less than or equal to that under CRE, with the difference  nondecreasing in $R^2$.  
\end{theorem}

\subsection{Sampling variance estimation and confidence intervals}\label{sec:var_ci_maha}
Asymptotic sampling variance and quantile range for $\hat{\tau}_Y$ depend on $V_{\tau\tau}$ and $R^2$, which are determined by the finite population covariances among potential outcomes and covariates. To obtain a sampling variance estimator and to construct an asymptotic confidence interval for $\tau_Y$, we need to estimate these finite population variances and covariances.
Let $s_{Y(z)}^2$, $s_{Y(z)\mid \bm{X}}^2$ and $\bm{s}_{Y(z),\bm{X}}$ be the sample variance of $Y$, sample variance of linear projection of $Y$ on $\bm{X}$, and sample covariance of $Y$ and $\bm{X}$ in treatment arm $z.$ We show in the Supplementary Material that under ReM they are asymptotically unbiased for their population analogues $S_{Y(z)}^2, S_{Y(z)\mid \bm{X}}^2$ and $S_{Y(z),\bm{X}}$.
Therefore, we can estimate $V_{\tau\tau}$ by \citep{ding2016decomposing}  
$$
\hat{V}_{\tau\tau}=r_1^{-1}s_{Y(1)}^2 + r_0^{-1}s_{Y(0)}^2 - 
(\bm{s}_{Y(1),\bm{X}}-\bm{s}_{Y(0),\bm{X}}) (\bm{S}_{\bm{X}}^2)^{-1}(\bm{s}_{\bm{X},Y(1)}-\bm{s}_{\bm{X},Y(0)}).
$$ 
We then estimate $R^2$ by 
\begin{align}
\label{eq:R2_hat_maha}
\hat{R}^2 = & \hat{V}_{\tau\tau}^{-1}
\left\{r_1^{-1}s_{Y(1)\mid \bm{X}}^2+r_0^{-1}s_{Y(0)\mid \bm{X}}^2
- \left(\bm{s}_{Y(1),\bm{X}}-\bm{s}_{Y(0),\bm{X}}\right) \left(\bm{S}_{\bm{X}}^2\right)^{-1}\left(\bm{s}_{\bm{X},Y(1)}-\bm{s}_{\bm{X},Y(0)}\right)
\right\}.
\end{align}
We set $\hat{R}^2$ to be $0$ if the estimator in (\ref{eq:R2_hat_maha}) is negative.

According to (\ref{eq:symp_var_maha}), we can estimate the sampling
variance of $\hat{\tau}_Y$ by $\hat{V}_{\tau\tau}  \{ 1- (1-v_{K,a})\hat{R}^2  \}/n,$
and according to (\ref{eq:asym_qr_maha}), we can construct a large sample $(1-\alpha)$ confidence interval for ${\tau}_Y$ using
$
\hat{\tau}_Y -  
\QR_{\alpha}(\hat{V}_{\tau\tau}, \hat{R}^2) / \sqrt{n}. 
$
The sampling variance estimator is smaller than \citet{Neyman:1923}'s sampling variance estimator for CRE, and the confidence interval is shorter than \citet{Neyman:1923}'s confidence interval for CRE. 
Not surprisingly, unless the residual from the linear projection of individual causal effect on the covariates is constant, 
the above sampling variance estimator and confidence interval are both asymptotically conservative,
in the sense that the probability limit of variance estimator is larger than or equal to the actual sampling variance, and the limit of coverage probability of confidence interval is larger than or equal to $(1-\alpha)$. 
Therefore, if we conduct ReM in the design stage but analyze data as in CRE, the consequential sampling variance estimator and confidence intervals will be overly conservative. These results are all intuitive, and we present the algebraic details for the proofs of these results in the Supplementary Material, where the unimodality of $L_{K,a}$ plays an important role in the conservativeness of confidence intervals. Interestingly, as shown in the Supplementary Material, we do not need more moment conditions beyond Condition \ref{cond:fp} to ensure the asymptotic properties of the variance estimator and the confidence intervals.

We also conduct simulations in the Supplementary Material with non-additive and additive causal effects, where the results agree with our theory for ReM.

\section{Rerandomization with tiers of covariates}
\label{sec::ReT}

\subsection{Mahalanobis distance with tiers of covariates criterion}

When covariates are thought to have different levels of importance for the outcomes, \citet{morgan2015rerandomization} proposed rerandomization using the Mahalanobis distance with differing criteria for different tiers of covariates. We partition the covariates into $T$ tiers indexed by $t=1,\ldots,T$ with decreasing importance, with $k_t$ covariates in tier $t$. Let $\bm{X}_i=(\bm{X}_i[1], \ldots, \bm{X}_i[T])$,  where $\bm{X}_i[t]$ denotes the covariates in tier $t$. Define $\bm{X}_i[\overline{t}] =  (\bm{X}_i[1], \ldots, \bm{X}_i[t])$, the covariates in the first $t$ tiers. Following the notation in \citet{morgan2015rerandomization}, we let $\bm{S}_{\bm{X}[\overline{t-1}]}^2$ be the finite population covariance matrix of the covariates in first $t-1$ tiers, and $\bm{S}_{\bm{X}[t], \bm{X}[\overline{t-1}]}$ be the finite population covariance matrix between $\bm{X}[t]$ and $\bm{X}[\overline{t-1}]$.
We first apply a block-wise Gram--Schmidt orthogonalization to the covariates to create the orthogonalized covariates:
\begin{eqnarray*}
\bm{E}_i[1] &=& \bm{X}_i[1] ,\\
\bm{E}_i[t]  &=& \bm{X}_i[t] - \bm{S}_{\bm{X}[t], \bm{X}[\overline{t-1}]} \left( \bm{S}_{\bm{X}[\overline{t-1}]}^2 \right)^{-1}\bm{X}_i[\overline{t-1}], \quad ( 2\leq t\leq \Tau ) 
\end{eqnarray*}
where $\bm{E}_i[t]$ is the residual of the projection of the covariates $\bm{X}_i[t]$ in tier $t$ onto the space spanned by the covariates in previous tiers;
 $\bm{E}_i=(\bm{E}_i[1],\ldots, \bm{E}_i[T])$. 
Let $\hat{\bm{\tau}}_{\bm{E}[t]}$ be the difference-in-means vector of $\bm{E}_i[t]$ between treatment and control groups, and $\bm{S}_{\bm{E}[t]}^2$ the finite population covariance matrix of $\bm{E}_i[t]$.
The Mahalanobis distance in tier $t$ is  
$$
M_t =  \frac{n_1n_0}{n}  \hat{\bm{\tau}}_{\bm{E}[t]}' \left( \bm{S}_{\bm{E}[t]}^2 \right)^{-1}\hat{\bm{\tau}}_{\bm{E}[t]},
$$
and rerandomization using the Mahalanobis distance with tiers of covariates (ReMT) accepts those treatment assignments with $M_t\leq a_t$, where $a_t$'s are predetermined constants $(1\leq t\leq \Tau)$. We can show that the criterion depends only on $\sqrt{n}\hat{\bm{\tau}}_{\bm{X}}$ and $\bm{V}_{\bm{xx}}$. If $T=1$, then ReMT is simply ReM.
We use $\mathcal{T}$ to denote the acceptance region for $\sqrt{n}\hat{\bm{\tau}}_{\bm{X}}$ under ReMT. The theory below extends \citet{morgan2015rerandomization} using the concepts from our Section \ref{sec::ReM}.

\subsection{Multiple correlation between potential outcomes and covariates with tiers}\label{sec:multi_corr_tier}

Similar to Section \ref{sec::R2}, we define the finite population squared multiple correlation 
between the potential outcome $Y(z)$ and the orthogonalized covariates in tier $t$ as $\rho^2_t(z)$,
and the finite population squared multiple correlation 
between the individual causal effect and the orthogonalized covariates in tier $t$ as $\rho^2_t(\tau)$.
We further define an $R^2$-type measure as the function of these finite population quantities and the proportions of group sizes:
$$
\rho_{t}^2  = \frac{S^2_{Y(1)}}{r_1 V_{\tau \tau}}\rho_t^2(1) + 
\frac{S^2_{Y(0)}}{r_0V_{\tau \tau}}\rho_t^2(0) - 
\frac{S^2_{\tau}}{V_{\tau \tau}}\rho_t^2(\tau), \quad  (1\leq t\leq T)
$$
which under the additive causal effect assumption reduces to
$
\rho_{{t}}^2 = 
\rho_{{t}}^2(1) = \rho_{{t}}^2(0),
$
the squared multiple correlation between $\bm{E}[{t}]$ and $Y(1)$ or $Y(0)$.

Under CRE, $\rho^2_t$ is the sampling squared multiple correlation between $\hat{\tau}_Y$ and  $\hat{\bm{\tau}}_{\bm{E}[t]}$, and can be equivalently written as 
\begin{align*}
\rho_{t}^2 & = 
\text{Corr}(\hat{\tau}_Y, \hat{\bm{\tau}}_{\bm{E}[t]})
=\frac{r_1^{-1}S_{Y(1)\mid \bm{E}[t]}^2+r_0^{-1}S_{Y(0)\mid \bm{E}[t]}^2
- S_{\tau\mid \bm{E}[t]}^2
}{{r_1^{-1}}S_{Y(1)}^2 + {r_0^{-1}}S_{Y(0)}^2 - S_{\tau}^2} ,  \quad  (1\leq t\leq T)
\end{align*} 
where $S_{Y(z)\mid \bm{E}[t]}^2$ and $S_{\tau \mid \bm{E}[t]}^2$ are the finite population variances of the projections of the potential outcomes and individual causal effects on the orthogonalized covariates in tier $t.$ For descriptive simplicity, we introduce $\rho^2_{T+1} = 1- \sum_{t=1}^T \rho^2_t = 1- R^2$ for later discussion.

\subsection{Asymptotic distribution of $\hat{\tau}_Y$}

The weak convergence of $\sqrt{n}(\hat{\tau}_Y-\tau_Y, \hat{\bm{\tau}}_{\bm{X}}')$ in (\ref{eq:asym_dist_gen}) still holds 
for ReMT, with region $\mathcal{M}$ replaced by region $\mathcal{T}.$
Intuitively, $\sqrt{n}(\hat{\tau}_Y -\tau_Y)$ can be decomposed into $(T+1)$ parts: the part unrelated to covariates and the $T$ projections onto the space spanned by the orthogonalized covariates in $T$ tiers.
Due to the construction of the orthogonalized covariates, these $(T+1)$ parts are orthogonal to each other and the constraint for balance on the Mahalanobis distance in tier $t$ affects only the $t$-th projection.

As earlier, let $\varepsilon_0\sim \mathcal{N}(0,1)$, and extending earlier notation, let $L_{k_t,a_t}\sim D_{t1}\mid \bm{D}_t'\bm{D}_t\leq a_t$, where $\bm{D}_t = ({D}_{t1},\ldots,{D}_{tk_t}) \sim \mathcal{N}(\bm{0}, \bm{I}_{k_t})$ for $1\leq t \leq T$.
\begin{theorem}\label{thm:asymp_dist_Y_maha_tier}
Under ReMT and Condition \ref{cond:fp},
\begin{align}\label{eq:asymp_dist_Y_maha_tier}
\sqrt{n}(\hat{\tau}_Y-\tau_Y) \mid  \sqrt{n}\hat{\bm{\tau}}_{\bm{X}} \in \mathcal{T} \ \ & \apprsim \ \  
\sqrt{V_{\tau\tau}}
\left(
\rho_{\Tau+1} \cdot  \varepsilon_0 + 
\sum_{t=1}^{T}
\rho_{t} \cdot 
L_{k_t,a_t}
\right),
\end{align}
where $(\varepsilon_0, L_{k_1,a_1}, \ldots, L_{k_T,a_T})$ are mutually independent.
\end{theorem}

Obviously, in \eqref{eq:asymp_dist_Y_maha_tier}, $\varepsilon_0$ is the part of $\sqrt{n}(\hat{\tau}_Y-\tau_Y)$ that is unrelated to the covariates, and $L_{k_t,a_t}$ is the part related to the orthogonalized covariates  $\bm{E}_i[t]$ in tier $t$. According to Proposition \ref{prop::curious}, the distribution in Theorem \ref{thm:asymp_dist_Y_maha_tier} involves  distributions that are easy to simulate.

\subsection{Asymptotic unbiasedness, sampling variance and quantile ranges}
Theorem \ref{thm:asymp_dist_Y_maha_tier} characterizes the asymptotic behavior of $\sqrt{n}(\hat{\tau}_Y -\tau_Y) $ under 
ReMT, which extends \citet{morgan2015rerandomization} as follows. 

First, the asymptotic distribution in (\ref{eq:asymp_dist_Y_maha_tier}) is symmetric around 0, implying that $\hat{\tau}_Y$ is asymptotically unbiased for ${\tau}_Y$. Therefore, all observed or unobserved covariates have asymptotically balanced means.

\begin{corollary}\label{cor:unbiased_tier}
Under ReMT and Condition \ref{cond:fp}, 
$
\mathbb{E}_{\text{a}}\left\{
\sqrt{n}(\hat{\tau}_Y - \tau_Y) \mid
\sqrt{n}\hat{\bm{\tau}}_{\bm{X}} \in \mathcal{T}
\right\} =0.
$
\end{corollary}

The asymptotic sampling variance of $\hat{\bm{\tau}}_{\bm{X}}$ under ReMT has a complicated but conceptually obvious form, and we give it in the Supplementary Material. Below we present only the PRIASV of $\hat{\tau}_Y$;  the PRIASVs for covariates are special cases of the same corollary because covariates are formally ``outcomes'' unaffected by the treatment. Recall the definition of $v_{k_t,a_t}=P(\chi^2_{k_t+2}\leq a_t)/P(\chi^2_{k_t}\leq a_t)$.

\begin{corollary}
\label{coro::ReT-variance}
Under ReMT and Condition \ref{cond:fp}, the asymptotic sampling variance of $\hat{\tau}_Y$ is 
\begin{align}\label{eq:asym_var_tier}
\Var_{\text{a}}\left\{ \sqrt{n}(\hat{\tau}_Y-\tau_Y) \mid \sqrt{n}\hat{\bm{\tau}}_{\bm{X}}  \in \mathcal{T}\right\} 
= & V_{\tau\tau}\left\{
1 - \sum_{t=1}^{T}(1-v_{k_t, a_t})\rho^2_{t}
\right\},
\end{align}
and the PRIASV of $\hat{\tau}_Y$ is
$
\sum_{t=1}^{T}(1-v_{k_t,a_t})\rho_{t}^2.
$
\end{corollary}

When the causal effect is additive, $\rho_{{t}}^2$ becomes the finite population squared multiple correlation between $\bm{E}[{t}]$ and $Y(0)$. Therefore, Corollary \ref{coro::ReT-variance} is an asymptotic extension of \citet[][Theorem 4.2]{morgan2015rerandomization}. 
When the thresholds $a_t$'s are close to zero, the asymptotic sampling variance on the right hand side of \eqref{eq:asym_var_tier} reduces to $V_{\tau\tau}(1-\sum_{t=1}^T \rho_t^2) = V_{\tau\tau}(1-R^2)$, which is identical to that of the regression adjusted estimator under CRE \citep{lin2013}.

We now compare the quantile range under ReMT to that under  CRE. Let $\nu_{\xi}(\rho_{1}^2, \rho_{2}^2, \ldots, \rho_{\Tau}^2)$ be the $\xi$th quantile of
$
\rho_{\Tau+1} \varepsilon_0 + 
\sum_{t=1}^{T}
\rho_{t}
L_{k_t,a_t}.
$
Although $\nu_{\xi}(\rho_{1}^2, \rho_{2}^2, \ldots, \rho_{\Tau}^2)$ depends also on $p_{a_t}$ and $ k_t$ ($1\leq t\leq K$), we omit them to avoid
notational clatter.
The $(1-\alpha)$ quantile range of the asymptotic distribution of $\sqrt{n}(\hat{\tau}_Y-\tau_Y)$ under ReMT is
\begin{align}\label{eq:asymp_qr_maha_tier}
\QR_\alpha(V_{\tau\tau}, \rho_{1}^2, \ldots, \rho_{\Tau}^2)
 = \left[  \nu_{\alpha/2}(\rho_{1}^2, \ldots, \rho_{\Tau}^2)\sqrt{V_{\tau\tau}},\ \  \nu_{1-\alpha/2}(\rho_{1}^2, \ldots, \rho_{\Tau}^2)\sqrt{V_{\tau\tau}}
\right].
\end{align}

The stronger the correlation between the outcome and the orthogonalized covariates
in tier $t$, the more reduction in quantile range we have when using ReMT rather than CRE.
The following theorem is immediate.

\begin{theorem}\label{thm:qr_reduct_maha_tier}
Under Condition \ref{cond:fp}, the $(1-\alpha)$ quantile range of the asymptotic distribution of $\sqrt{n}(\hat{\tau}_Y-\tau_Y)$ under ReMT is narrower than, or equal to the one under CRE, and the reduction in length of the quantile range is nondecreasing in $\rho_{t}^2$ for all $1\leq t\leq \Tau$.
\end{theorem}

\subsection{Sampling variance estimation and confidence interval}\label{sec:est_gen}

We can estimate $V_{\tau\tau}$ and  $\rho_{{t}}^2\ (1\leq t\leq T)$ in the same way as in ReM, and we estimate $\rho^2_{T+1}$ by $1-\hat{R}^2.$  In practice, we set $\hat{\rho}^2_{t}\ (1\leq t\leq T)$ to $0$ when it is negative due to sampling variability, and standardize their sum to $\hat{R}^2.$ 
According to (\ref{eq:asym_var_tier}) and (\ref{eq:asymp_qr_maha_tier}), we can estimate the sampling variance of $\hat{\tau}_Y$ and $(1-\alpha)$ confidence intervals for $\tau_Y$ by replacing the unknown quantities with their point estimates. The sampling variance estimator is smaller than \citet{Neyman:1923}'s sampling variance estimator for CRE, and the confidence interval is shorter than \citet{Neyman:1923}'s confidence interval for CRE; both are asymptotically conservative in general, and only when the residual from the linear projection of individual causal effect on the covariates is constant, are they asymptotically exact. Therefore, analyzing data from ReMT as from CRE, the resulting sampling variance estimator and confidence intervals are overly conservative. These intuitive statements appear to require notationally lengthy proofs, which are relegated to the Supplementary Material. Specifically, the proof for the conservativeness of confidence intervals utilizes the unimodality of the $L_{k_t,a_t}$'s.

\section{Rerandomization with more general balance criterion}\label{sec:ReG}
\subsection{More general balance criterion}\label{sec:gen_bal_cri}

As pointed out by \citet{morgan2012rerandomization},  the criterion can be any accept-reject function of the treatment assignment and covariate balance. 
We can always use randomization tests for a sequence of sharp null hypotheses, and thereby construct fiducial confidence intervals by inverting these randomization tests (e.g., under the additive causal effects assumption). 
In this section, we discuss the repeated sampling properties of the difference-in-means estimator, where we
consider covariate balance criteria that depend only on $\sqrt{n}\hat{\bm{\tau}}_{\bm{X}}$ and $\bm{V}_{\bm{xx}}$, including ReM and ReMT as special cases, and write the binary covariate balance indicator function as $\phi(\sqrt{n}\hat{\bm{\tau}}_{\bm{X}}, \bm{V}_{\bm{xx}})$. Let $\mathcal{G}$ denote the acceptance region for rerandomization with the general covariate balance criterion $\phi$ (ReG),
i.e., $\sqrt{n}\hat{\bm{\tau}}_{\bm{X}} \in \mathcal{G}=\left\{
\bm{\mu}: \phi(\bm{\mu}, \bm{V}_{\bm{xx}}) = 1
\right\}$.

For technical reasons, we require $\phi$ to satisfy the following conditions. First, $\phi$ is almost surely continuous. Second, 
$\Var(\bm{B}\mid \phi(\bm{B},\bm{V}_{\bm{xx}})=1)$, as a function of $\bm{V}_{\bm{xx}}$ with $\bm{B} \sim \mathcal{N}(\bm{0},\bm{V}_{\bm{xx}})$, is continuous for all $\bm{V}_{\bm{xx}}>0$. Third,  $P(\phi(\bm{B}, \bm{V}_{\bm{xx}})=1)>0$, for all $\bm{V}_{\bm{xx}}>0$ with $\bm{B} \sim \mathcal{N}(\bm{0},\bm{V}_{\bm{xx}})$. Fourth, $\phi(\bm{\mu},\bm{V}_{\bm{xx}})=\phi(-\bm{\mu},\bm{V}_{\bm{xx}})$ for all $\bm{\mu}$ and $\bm{V}_{\bm{xx}}>0$. The first two conditions impose certain smoothness on $\phi$, and the third condition prevents the acceptance region from being a set of measure zero.
The fourth condition imposes symmetry considerations, because relabeling the treatment and control units should not change the balance. 
Both ReM and ReMT satisfy these conditions. Below, we summarize theory in parallel with Sections \ref{sec::ReM} and \ref{sec::ReT}.

\subsection{Asymptotic sampling properties}
\label{sec:inf_ReG}

The weak convergence in (\ref{eq:asym_dist_gen}) holds with $\mathcal{M}$ replaced by $\mathcal{G}.$ The projection of 
$A$ onto $\bm{B}$ is $\bm{V}_{\tau \bm{x}}\bm{V}_{\bm{xx}}^{-1}\bm{B}$, and the residual is $\varepsilon=A - \bm{V}_{\tau \bm{x}}\bm{V}_{\bm{xx}}^{-1}\bm{B}$. Therefore, the asymptotic distribution of $\hat{\tau}_Y$ is 
\begin{align}\label{eq:asym_gen_re}
\sqrt{n}(\hat{\tau}_Y -\tau_Y) \mid \sqrt{n}\hat{\bm{\tau}}_{\bm{X}} \in\mathcal{G} \apprsim  \varepsilon + \bm{V}_{\tau \bm{x}}\bm{V}_{\bm{xx}}^{-1}\bm{B} \mid \bm{B} \in \mathcal{G},
\end{align}
where $\varepsilon \sim \mathcal{N}(0, V_{\tau\tau}(1-R^2))$ is independent of $\bm{B}\sim \mathcal{N}(\bm{0},  \bm{V}_{\bm{xx}})$.
The asymptotic distribution in \eqref{eq:asym_gen_re}
can be easily simulated. 

The symmetry condition $\phi(\bm{\mu},\bm{V}_{\bm{xx}})=\phi(-\bm{\mu},\bm{V}_{\bm{xx}})$ implies that the distribution in (\ref{eq:asym_gen_re}) is symmetric around $0$, which further implies that $\hat{\tau}_Y$ is asymptotically unbiased for $\tau_Y$. Viewing covariates as outcomes unaffected by the treatment, all observed or unobserved covariates asymptotically have the same means in treatment and control groups.

\subsection{Advice for the investigator}

Because $\bm{V}_{\bm{xx}}$ is known in finite population inference, we can choose ReGs that result in better covariate balance before the physical experiments. Because $\Var_{\text{a}}\left\{ \sqrt{n} \hat{\tau}_{\bm{X}} \mid \sqrt{n}\hat{\bm{\tau}}_{\bm{X}} \in \mathcal{G} \right\} = \Var(\bm{B} \mid \bm{B} \in \mathcal{G})$, it is important to check whether $\Var(\bm{B} \mid \bm{B} \in \mathcal{G}) \leq  \Var(\bm{B})$ holds, i.e., $\Var(\bm{B}) - \Var(\bm{B} \mid \bm{B} \in \mathcal{G}) $ is positive semi-definite.
This ensures that ReG using acceptance region $\mathcal{G}$ reduces the sampling covariance matrix of the difference-in-means of the covariates. Otherwise, we need to change the criterion. If 
$
\bm{V}_{\bm{xx},\phi}   \equiv   \Var(\bm{B} \mid \bm{B} \in \mathcal{G})\leq \Var(\bm{B})=\bm{V}_{\bm{xx}}
$
holds, 
then we can derive the PRIASV  for all observed covariates under ReG. From the decomposition 
\begin{align*} 
\Var_{\text{a}}\left\{ \sqrt{n}(\hat{\tau}_Y-\tau_Y) \mid \sqrt{n}\hat{\bm{\tau}}_{\bm{X}} \in \mathcal{G} \right\} 
& = {V}_{\tau\tau}(1-R^2) + \bm{V}_{\tau \bm{x}}\bm{V}_{\bm{xx}}^{-1} 
\bm{V}_{\bm{xx},\phi}
\bm{V}_{\bm{xx}}^{-1}\bm{V}_{\bm{x}\tau},
\end{align*}
we can derive 
the PRIASV of $\sqrt{n}(\hat{\tau}_Y-\tau_Y)$ under ReG.
However, without imposing further conditions on $\phi$, there is no guarantee that the asymptotic quantile range of $\hat{\tau}$ will be shorter under ReG than that under CRE. Therefore, in the design stage of the experiment, we recommend choosing balance criteria that are expected to lead to both variance and quantile range reductions, such as ReM and ReMT.

\subsection{Sampling variance estimation and confidence interval}
We can show that $s_{Y(z)}^2$, $s_{Y(z)\mid \bm{X}}^2$ and $\bm{s}_{Y(z),\bm{X}}^2$ are asymptotically unbiased for $S_{Y(z)}^2$, $S_{Y(z)\mid \bm{X}}^2$ and $\bm{S}_{Y(z),\bm{X}}^2$ under ReG. Therefore, we can unbiasedly estimate $V_{\tau\tau}$ by $\hat{V}_{\tau\tau}$ as earlier,
and $R^2$ by $\hat{R}^2$ in the same form as (\ref{eq:R2_hat_maha}).
We can then estimate $\bm{V}_{\tau \bm{x}}=r_1^{-1}\bm{S}_{Y(1), \bm{X}} + r_0^{-1}\bm{S}_{Y(0), \bm{X}}$ and the sampling variance of $\widehat{\tau}_Y$ by replacing the unknown quantities with their point estimates.
We can then estimate the asymptotic distribution in (\ref{eq:asym_gen_re}), as well as its quantile ranges.
If ${\bm{V}}_{\tau \bm{x}}\bm{V}_{\bm{xx}}^{-1}\bm{B} \mid \bm{B} \in \mathcal{G}$ is unimodal, as in ReM and ReMT, the final confidence interval is asymptotically conservative. 
We relegate more details to the Supplementary Material.

\section{An education example with tiers of covariates}\label{sec:edu_eg}
We illustrate our theory using the data from the Student Achievement and Retention Project \citep{angrist2009incentives}, a randomized evaluation of academic services and incentives at one of the satellite campuses of a large Canadian university, involving college freshmen. A treatment group of 150 students was offered an array of support services and substantial cash awards for meeting a target first year grade point average (GPA), and a control group of many more (1006) students received only standard university support services.


To illustrate the benefit of rerandomization, we use the 15 covariates as listed in Table \ref{tab:student_cov}, and exclude students with missing values, resulting in 118 students in the treatment group and 856 in the control.
To make the simulation relevant to the real data, we fix unknown parameters based on some simple model fitting: We  fit a linear regression of the observed first year GPA on the treatment indicator, all covariates and their interactions,
and use the fitted model to generate all potential outcomes under non-additivity.  
Note that the generating models for the potential outcomes are not linear in the covariates themselves. 
To make the data generating process realistic, we simulate eight pseudo sets of potential outcomes using the fitted model with different choices for the variance of the residuals. 
The error terms for $Y(1)$ and $Y(0)$ are independent, and therefore conditional on the covariates, the potential outcomes are simulated as uncorrelated, but they have a positive correlation marginally.
The final potential outcomes are all truncated to lie on $[0,4]$, mimicking the value of the GPA. We choose different variances of residuals such that the values of $R^2$ for the eight simulated data sets are located approximately evenly within interval $[0,0.5]$. One choice for the variance of residuals is the one estimated from the fitted linear model, and the corresponding $R^2$ is about $0.23$.

\begin{table}
\centering
\caption{Covariates in the Student Achievement and Retention Project.
The numbers of covariates in these three tiers are 1, 4 and 10, and the thresholds are $a_1 = 0.016$, $a_2 = 1.064$ and $a_3 = 4.865.$ }
\label{tab:student_cov}
\begin{tabular}{|c|c|}
\hline
Tier  & Covariates\\
\hline
 1  & high school GPA\\
\hline
 2  & whether lives at home, gender, age,\\
&  whether rarely puts off studying for tests\\
\hline
 3  & whether mother is a college graduate, whether mother is a high school graduate,  \\
&  mother tongue (English or other), whether plans to work while in school,  \\
&  whether father is a college graduate, whether father is a high school graduate,\\
&  whether never puts off studying for tests, whether wants more than a bachelor degree,\\
&  whether intends to finish in 4 years,  whether at the first choice school\\
\hline
\end{tabular}
\end{table}

Table \ref{tab:student_cov} partitions the covariates into three tiers with decreasing a priori importance to the outcome. 
As suggested by \citet{morgan2015rerandomization}, for tiers with increasing numbers of covariates, 
we choose $a_t$ such that $P(\chi^2_{k_t}\leq a_t) = (0.001)^{1/3} = 0.1$ for $t=1,2,3$.
We simulate data under ReMT, and obtain the confidence intervals based on our asymptotic theory for ReMT and \citet{Neyman:1923}'s results for CRE. 
Figure \ref{fig:star_coverage_all} shows the empirical coverage probabilities of our and \citet{Neyman:1923}'s confidence intervals, showing that \citet{Neyman:1923}'s CRE confidence intervals are highly conservative.
Note that there are $15$ covariates and only $118$ units in the treatment group, and
the sample size is not extremely large. Despite this, our asymptotic confidence interval works well in this example.

To evaluate the performance of ReMT compared to CRE, we compare the average length of \citet{Neyman:1923}'s confidence interval under CRE with the confidence interval under ReMT. From Figure \ref{fig:star_ci_length}, the percentage reduction in average lengths of the confidence intervals under ReMT compared to \citet{Neyman:1923}'s under CRE is nondecreasing in $R^2$.
We also compare the empirical $95\%$ quantile range of $\hat{\tau}_Y$ under ReMT and CRE, and the percentage reduction in the lengths of quantile ranges are close to the percentage reduction for average lengths of confidence intervals.
When $R^2$ is close to that of the real data set (i.e. $0.23$), the percentage increase in the effective sample size, that is, 
the sample size needed in CRE in order for $\hat{\tau}_Y$ to achieve the same $95\%$ quantile range under ReMT, is about $24\%$. When $R^2$ is about twice as large as with the real data (i.e. 0.5), the percentage increase in the effective sample size increases to $80\%$.

\begin{figure}[htb]
\centering
\begin{subfigure}{.5\textwidth}
  \centering
  \includegraphics[width=0.8\linewidth]{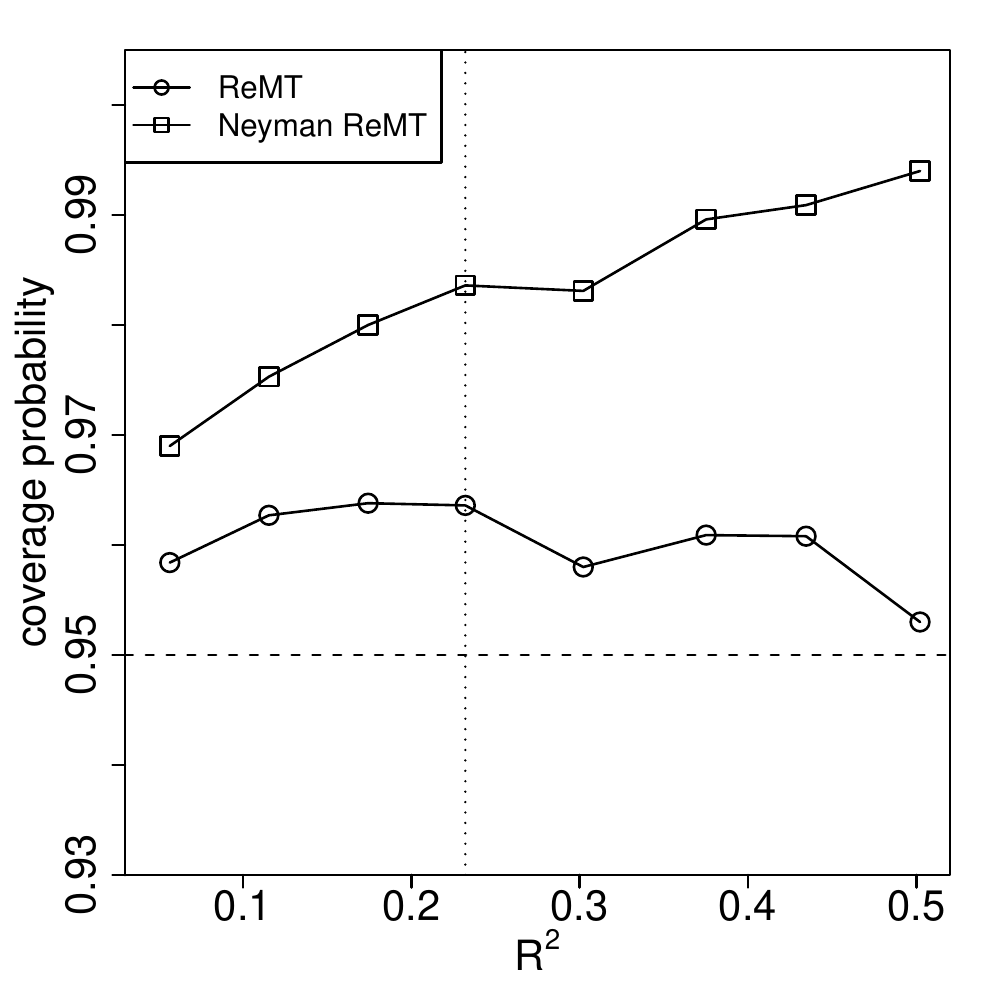}
   \caption{
    \centering Empirical coverage probabilities of our and \citet{Neyman:1923}'s $95\%$ confidence intervals under ReMT}
  \label{fig:star_coverage_all}
\end{subfigure}%
\begin{subfigure}{.5\textwidth}
  \centering
  \includegraphics[width=.8\linewidth]{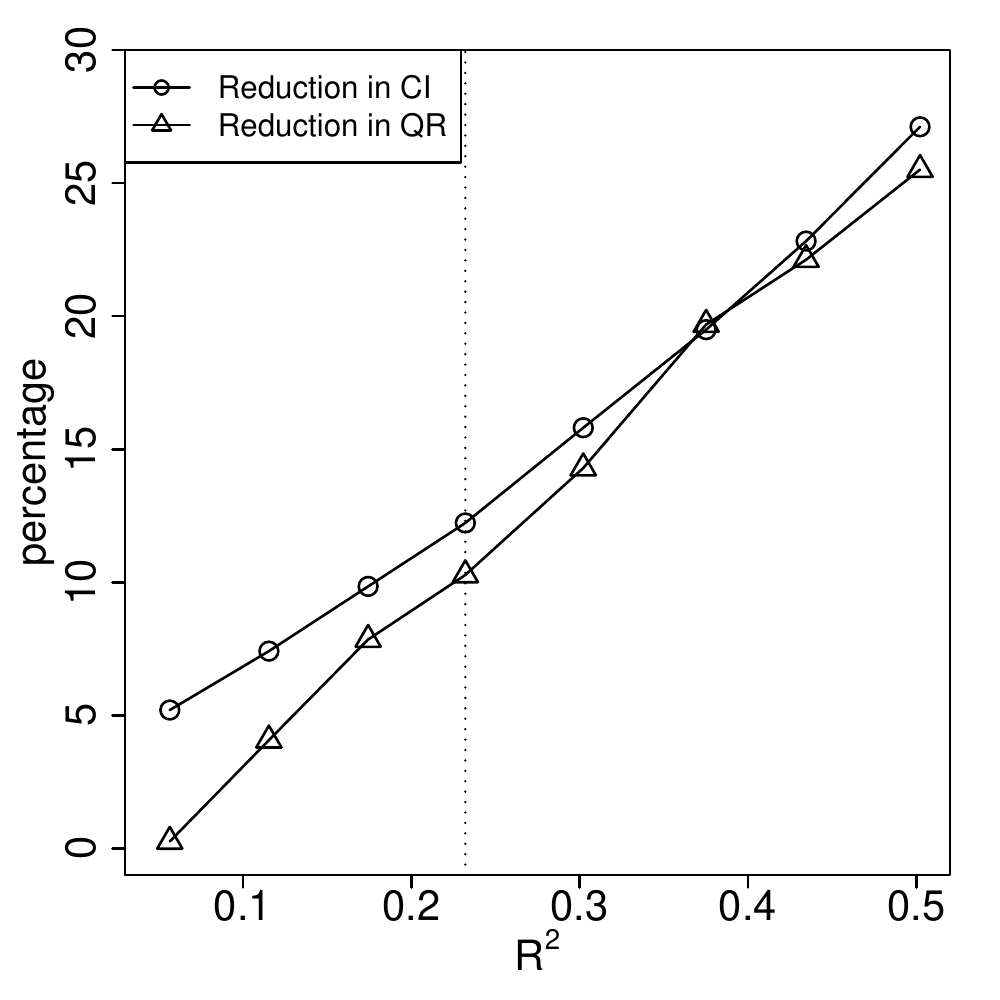}
   \caption{
    \centering Percentage reductions of average lengths of confidence intervals ($\bigcirc$) and quantile ranges ($\triangle$) comparing ReMT with CRE}
    \label{fig:star_ci_length}
\end{subfigure}
\caption{Eight data sets simulated based on the Student Achievement and Retention Project}
\end{figure}

\section{Conclusions, Connections and Extensions}\label{sec:conclude}
Extending \citet{morgan2012rerandomization,morgan2015rerandomization},
we show, using analysis and simulations, that rerandomization balances covariates better than complete randomization, and provides a more precise difference-in-means estimator for the average causal effect. 
The asymptotic distributions of the difference-in-means estimator under rerandomization with strigent constraints are close to that of the regression adjusted estimator under CRE \citep{lin2013}, implying that rerandomization does the covariate adjustment in the design stage and avoids outcome modeling. The new asymptotic distributions allow us to construct confidence intervals for the average causal effect, when the classical \citet{Neyman:1923}'s inference for CRE is overly conservative.

\section*{Appendix}
In this Appendix, we provide proofs for the asymptotic distribution of $\sqrt{n}(\hat{\tau}_Y-\tau_Y)$ under ReM and ReMT, and the representation for random variable $L_{K,a}$. 
First, we need the following two lemmas.

\begin{lemma}\label{lemma:L}
Let $L_{K,a}\sim D_1\mid \bm{D}'\bm{D}\leq a$, where $\bm{D}=(D_1,\ldots,D_K)'\sim \mathcal{N}(\bm{0}, \bm{I}_K)$. For any $K$ dimensional unit vector $\bm{h}$,  we have 
$
L_{K,a} \sim \bm{h}'\bm{D} \mid \bm{D}'\bm{D}\leq a.
$
\end{lemma}
\begin{proof}[Proof of Lemma \ref{lemma:L}]
We construct an orthogonal matrix $\bm{H}$ whose first row is $\bm{h}'$. Then $\bm{D}\sim \bm{H}\bm{D}$. Therefore, $L_{K,a} \sim D_1 \mid \bm{D}'\bm{D}\leq a \sim \bm{h}'\bm{D} \mid (\bm{H}\bm{D})'\bm{H}\bm{D}\leq a\sim \bm{h}'\bm{D} \mid \bm{D}'\bm{D}\leq a$.
\end{proof}

\begin{lemma}\label{lemma:U_K}
Let $U_K$ be the first coordinate of the uniform random vector over the $(K-1)$ dimensional unit sphere. Let $S$ be a random sign taking $\pm 1$ with probability $1/2$,  $\beta_K \sim \text{Beta}\left(1/2, (K-1)/2\right)$ be a Beta random variable degenerating to a point mass at $1$ when $K=1$, and $(S, \beta_K)$ are mutually independent. Then
$
U_K \sim S\sqrt{\beta_K}.
$
\end{lemma}
\begin{proof}[Proof of Lemma \ref{lemma:U_K}]
When $K=1$, it is easy to see Lemma \ref{lemma:U_K} holds. 
When $K\geq 2$, 
let $\bm{D}=(D_1,D_2,\ldots,D_K)' \sim \mathcal{N}(\bm{0},\bm{I}_K)$.The standardized normal
random vector with unit length is uniformly distributed on the unit sphere \citep{fang1989symmetric}, and therefore
$U_K \sim D_1/\sqrt{\bm{D}'\bm{D}}$. 
Let $S$ be a random sign independent of $\bm{D}$; then $D_1\sim S|D_1|$. Thus $(D_1, D_2,\ldots,D_K)\sim (S|D_1|, D_2,\ldots,D_K)$, 
$$
U_K \sim  D_1/\sqrt{\bm{D}'\bm{D}} \sim  S|D_1|/\sqrt{\bm{D}'\bm{D}} = {S}\cdot\sqrt{\frac{D_1^2}{\sum_{k=1}^K D_k^2}}, 
$$
and $S$ is independent of ${D_1^2}/{(\sum_{k=1}^K D_k^2)}$.
Because $D_1^2, \ldots, D_K^2$ follow i.i.d. Gamma distributions with shape parameter $1/2$ and scale parameter $2$,
$
{D_1^2}/{(\sum_{k=1}^K D_k^2)}\sim \text{Beta}(1/2, (K-1)/2)
$  is independent of $S$. Therefore,
Lemma \ref{lemma:U_K} holds.
\end{proof}

\begin{proof}[{\bf Proof of Theorem \ref{thm:asymp_dist_Y_maha}}]
The linear projection of $A$ on $\bm{B}$ is $\bm{V}_{\tau \bm{x}}\bm{V}_{\bm{xx}}^{-1}\bm{B}$, which has variance $c^2=\bm{V}_{\tau \bm{x}}\bm{V}_{\bm{xx}}^{-1}\bm{V}_{\bm{x}\tau} = V_{\tau\tau} R^2.$
The residual from the linear projection of $A$ on $\bm{B}$ is 
$$
\varepsilon =A - \bm{V}_{\tau \bm{x}}\bm{V}_{\bm{xx}}^{-1}\bm{B} \sim \mathcal{N}(0, (1-R^2)V_{\tau\tau}) \sim \sqrt{V_{\tau\tau}(1-R^2)} \cdot \varepsilon_0,
$$ 
where $\varepsilon_0\sim \mathcal{N}(0,1)$. Let $\bm{h}'=\bm{V}_{\tau \bm{x}}\bm{V}_{\bm{xx}}^{-1/2}/c$ be the standardized vector of $\bm{V}_{\tau \bm{x}}\bm{V}_{\bm{xx}}^{-1/2}$ with unit length, and $\bm{D}=\bm{V}_{\bm{xx}}^{-1/2}\bm{B} \sim \mathcal{N}(\bm{0}, \bm{I}_K)$ be the standardization of $\bm{B}$. Then
$$
A = \varepsilon + \bm{V}_{\tau \bm{x}}\bm{V}_{\bm{xx}}^{-1}\bm{B} = \varepsilon + \bm{V}_{\tau \bm{x}}\bm{V}_{\bm{xx}}^{-1/2}\bm{D} = \varepsilon + c \bm{h}' \bm{D}.
$$
According to Proposition \ref{thm:asymp_dist_maha} and Lemma \ref{lemma:L},
\begin{align*}
&\sqrt{n}(\hat{\tau}_Y-\tau_Y) \mid  \sqrt{n}\hat{\bm{\tau}}_{\bm{X}} \in \mathcal{M} \ \  \apprsim \ \  A \mid \bm{B} \in \mathcal{M}\\
 \sim & \ \ 
\varepsilon + c \bm{h}' \bm{D} \mid \bm{D}'\bm{D}\leq a \ \ 
 \sim \ \ \varepsilon + c L_{K,a}
 \sim \ \ 
\sqrt{V_{\tau\tau}}\left( 
\sqrt{1-R^2}\cdot\varepsilon_0 + \sqrt{R^2} \cdot L_{K,a}   \right),
\end{align*}
where $\varepsilon$, or $\varepsilon_0$, is independent of $L_{K,a}$.
\end{proof}

\begin{proof}[{\bf Proof of Theorem \ref{thm:asymp_dist_Y_maha_tier}}]
We use $\bm{\Gamma}$ to denote the linear transformation from $\bm{X}_i$ to $\bm{E}_i$, i.e. $\bm{E}_i = \bm{\Gamma}\bm{X}_i$, where $\bm{\Gamma}$ depends only on $\bm{V}_{\bm{xx}}$.
Correspondingly, let
$\bm{G} = \bm{\Gamma} \bm{B} = (\bm{G}_{1}', \bm{G}_{2}', \ldots,\bm{G}_{T}')'$
be the block-wise Gram--Schmidt orthogonalization of $\bm{B}$, where $\bm{G}_{t}$ is a $k_t$ dimensional random vector. 
Let 
$\bm{D}_t = \Var(\bm{G}_{t})^{-1/2}\bm{G}_{t} \sim \mathcal{N}(\bm{0}, \bm{I}_{k_t})$ be the standardization of $\bm{G}_{t}$. The linear projection of $A$ on $\bm{G}_{t}$ is $\Cov(A, \bm{G}_{t})\Var(\bm{G}_{t})^{-1}\bm{G}_{t}$, with variance $c_{t}^2 = \Cov(A, \bm{G}_{t})\Var(\bm{G}_{t})^{-1}\Cov( \bm{G}_{t}, A)$. 
Because $(\hat{\tau}_Y, \hat{\bm{\tau}}_{\bm{E}}') = (\hat{\tau}_Y, \bm{\Gamma}\hat{\bm{\tau}}_{\bm{X}}')$ has the same covariance matrix as $(A,\bm{G}')=(A, \bm{\Gamma} \bm{B}')$,
the sampling variance of linear projection of $\hat{\tau}_Y$ on $\hat{\bm{\tau}}_{\bm{E}[t]}$ is the same as
  the variance of linear projection of $A$ on $\bm{G}_{t}$. 
The former is $V_{\tau\tau}\rho_{t}^2$, and the latter is $c_{t}^2$.    
Therefore, $c_{t}^2=V_{\tau\tau}\rho_{t}^2$.
The residual from the linear regression of $A$ on $\bm{B}$ (or equivalently on $\bm{G}$) is
\begin{eqnarray*}
\varepsilon & = & A -\bm{V}_{\tau \bm{x}}\bm{V}_{\bm{xx}}^{-1}\bm{B}
= A - \sum_{t=1}^\Tau \Cov(A, \bm{G}_{t})\Var(\bm{G}_{t})^{-1}\bm{G}_{t}\\
&\sim & \mathcal{N}\left(
0, V_{\tau\tau}\left( 1 - \sum_{t=1}^T \rho_{t}^2 \right)
\right)
\sim \mathcal{N}(0,V_{\tau\tau}\rho^2_{T+1})\sim \sqrt{V_{\tau\tau}} \cdot \rho_{T+1}\varepsilon_0,
\end{eqnarray*}
where $\varepsilon_0 \sim \mathcal{N}(0,1)$.
Let $\bm{h}_{t}'=\Cov(A, \bm{G}_{t})\Var(\bm{G}_{t})^{-1/2}/c_{t}$ be the standardized vector of $ \Cov(A, \bm{G}_{t})\Var(\bm{G}_{t})^{-1/2}$ with unit length. 
Then $A$ has the following decomposition:
\begin{eqnarray*}
A & =&  \varepsilon + \sum_{t=1}^\Tau \Cov(A, \bm{G}_{t})\Var(\bm{G}_{t})^{-1}\bm{G}_{t}
 =  \varepsilon + \sum_{t=1}^\Tau \Cov(A, \bm{G}_{t})\Var(\bm{G}_{t})^{-1/2}\bm{D}_t\\
& = & \varepsilon + \sum_{t=1}^\Tau c_{t}\bm{h}_{t}' \bm{D}_t.
\end{eqnarray*}
Because Proposition \ref{thm:asymp_dist_maha} holds for ReMT with $\mathcal{M}$ replaced by $\mathcal{T}$, and $(\bm{D}_1, \ldots, \bm{D}_T)$ are mutually independent, 
\begin{align*}
\sqrt{n}(\hat{\tau}_Y-\tau_Y) \mid  \sqrt{n}\hat{\bm{\tau}}_{\bm{X}} \in \mathcal{T}  & \apprsim A \mid \bm{B} \in \mathcal{T} \sim 
\varepsilon + \sum_{t=1}^\Tau c_{t} 
\left( \bm{h}_{t}'\bm{D}_t \mid \bm{D}_t'\bm{D}_t\leq a_t \right).
\end{align*}
According to Lemma \ref{lemma:L}, 
\begin{align*}
\sqrt{n}(\hat{\tau}_Y-\tau_Y) \mid  \sqrt{n}\hat{\bm{\tau}}_{\bm{X}} \in \mathcal{T}  \apprsim \varepsilon + \sum_{t=1}^T c_{t}L_{k_t,a_t}
 \sim 
\sqrt{V_{\tau\tau}}
\left(
\rho_{\Tau+1} \varepsilon_0 + 
\sum_{t=1}^{T}
\rho_{t}
L_{k_t,a_t}
\right),
\end{align*}
where $L_{k_t,a_t}\sim \bm{D}_{t1} \mid \bm{D}_t'\bm{D}_t\leq a_t$, $(\varepsilon, L_{k_1,a_1}, \ldots, L_{k_T,a_T})$ are mutually independent, and $(\varepsilon_0, L_{k_1,a_1}, \ldots, L_{k_T,a_T})$ are mutually independent.
Therefore, Theorem \ref{thm:asymp_dist_Y_maha_tier} holds.
\end{proof}

\begin{proof}[{\bf Proof of Proposition \ref{prop::curious}}]
Let $\bm{D}=(D_1,\ldots,D_K)' \sim \mathcal{N}(\bm{0}, \bm{I}_K)$ and $L_{K,a}\sim D_1\mid \bm{D}'\bm{D}\leq a$.
First, we show that $L_{K,a}$ is symmetric and unimodal around $0.$ Let $f(\cdot)$ be the density of standard Gaussian distribution. It is easy to show that $L_{K,a}\sim D_1 \mid \bm{D}'\bm{D}\leq a$ is symmetric around $0$, and has density
\begin{align*}
p(l) = \frac{f(l)P(\sum_{k=2}^K {D}_k^2 \leq a-l^2)}{P(\bm{D}'\bm{D}\leq a)}
= f(l)\frac{P(\chi^2_{K-1}\leq a-l^2)}{P(\chi^2_K\leq a)}
 \quad  (-\sqrt{a}< l <\sqrt{a}).
\end{align*}
Therefore, $L_{K,a}$ is unimodal, because for any $|l_1|<|l_2|<\sqrt{a}$, 
\begin{align*}
p(l_1) = & f(l_1)\frac{P(\chi^2_{K-1}\leq a-l_1^2)}{P(\chi^2_K\leq a)} \geq  f(l_2)\frac{P(\chi^2_{K-1}\leq a-l_2^2)}{P(\chi^2_K\leq a)}
=  p(l_2).
\end{align*}

Second, the variance formula for $L_{K,a}$ follows from \citet[][Theorem 3.1]{morgan2012rerandomization}.

Third, we represent $L_{K,a}$ by known distributions.  Let $R_K=\sqrt{\bm{D}'\bm{D}}$ be the length of vector $\bm{D}$,
 and $\bm{D}/R_K$ be the normalized vector of $\bm{D}$ with unit length. 
From the property of the multivariate Gaussian distribution, $R_K^2\sim \chi^2_K,$ $\bm{D}/R_K$ follows the uniform distribution on the $K-1$ dimensional unit sphere, and 
 they are independent \citep{fang1989symmetric}. Let $U_K$ be the first coordinate of $\bm{D}/R_K$, then 
$
L_{K,a}\sim U_K R_K \mid R_K^2\leq a.
$
Because $\chi_{K,a}\sim R_K \mid R_K^2\leq a$, and $\chi_{K,a}$ is independent of $U_K$, we have $L_{K,a}\sim U_K R_K \mid R_K^2\leq a\sim  \chi_{K,a} U_K$.
According to Lemma \ref{lemma:U_K}, we have $L_{K,a}\sim \chi_{K,a}U_K \sim \chi_{K,a} S\sqrt{\beta_K}$.
\end{proof}


\bibliographystyle{plainnat}
\bibliography{causal}

\newpage
\setcounter{page}{1}
\begin{center}
\bf \huge 
Supplementary Material
\end{center}

\bigskip

\setcounter{equation}{0}
\setcounter{section}{0}
\setcounter{figure}{0}
\setcounter{example}{0}
\setcounter{proposition}{0}
\setcounter{corollary}{0}
\setcounter{theorem}{0}
\setcounter{table}{0}

\renewcommand {\theproposition} {A\arabic{proposition}}
\renewcommand {\theexample} {A\arabic{example}}
\renewcommand {\thefigure} {A\arabic{figure}}
\renewcommand {\thetable} {A\arabic{table}}
\renewcommand {\theequation} {A\arabic{equation}}
\renewcommand {\thelemma} {A\arabic{lemma}}
\renewcommand {\thesection} {A\arabic{section}}
\renewcommand {\thetheorem} {A\arabic{theorem}}
\renewcommand {\thecorollary} {A\arabic{corollary}}

Section \ref{sec:numer_ex} uses simulations to evaluate the asymptotic approximations for the sampling distributions of $\hat{\tau}$, as well as the coverage probablilities of $95\%$ confidence intervals for $\tau$ under ReM. Section \ref{sec:weak_conv_proof} shows the weak convergence of $\sqrt{n}(\hat{\tau}_Y-\tau_Y, \hat{\bm{\tau}}_{\bm{X}}')$ under ReG, the asymptotic unbiasedness of $\hat{\tau}_Y$ and balance in means of all covariates, and the formula for the sampling squared multiple correlation between $\hat{\tau}_Y$ and $\hat{\bm{\tau}}_{\bm{X}}$ under CRE in Proposition \ref{prop:R_sampling_cor}. Section \ref{sec::improvement} shows the percentage reductions in asymptotic sampling variances and lengths of quantile ranges under rerandomization. Section \ref{sec::inference} shows the asymptotic conservativeness of sampling variance estimators and confidence intervals.

\section{Numerical Examples}\label{sec:numer_ex}

We conduct numerical examples where the group sizes are very different and the potential outcomes are simulated from a nonlinear model with binary values. Let
$r_1 = 0.1, r_0 = 0.9$ and $K=3$, so the two treatment group sizes are very different. Let $\bm{\beta}_1 = (2,3,4)'$ and $\bm{\beta}_0 = (0,1,1)'$. 
The covariates for all units are i.i.d samples from $ X_{k} \sim \text{Bernoulli}(0.5)$ for $1\leq k\leq K$, where $(X_1,\ldots,X_K)$ are mutually independent.
The binary potential outcomes are i.i.d samples from:
\begin{align}\label{eq:num_gen}
& Y(z) = \mathbb{I}\left\{ z + 
\bm{\beta}_z'(\bm{X}-0.5\bm{1}_K)
 + \delta_{z}\geq 0\right\}, \quad \delta_{0}, \delta_{1} \overset{\text{i.i.d}}{\sim} \mathcal{N}(0,1), \ \ 
(z=0,1),
\end{align}
where $\mathbb{I}(\cdot)$ is the indicator function.
We simulate three data sets with different sample sizes $(1000,3000,5000)$, and the causal effects for these simulated data sets are not additive. For example, when $n=1000,$ $\tau_Y = 0.121, S_{Y(1)}^2=0.24, S_{Y(0)}^2=0.25$, and $S_{\tau}^2=0.33$.
Figure \ref{fig:binary_sim} shows the histograms of $\sqrt{n}(\hat{\tau}_Y-\tau_Y)$ under both ReM with $p_a=0.001$ and CRE, based on $10^5$ rerandomizations and $10^5$ complete randomizations, 
as well as their asymptotic approximations using (\ref{eq:dist_maha_Gaussian_trun_norm}) and Gaussian distributions. Although the potential outcomes models are not linear, the asymptotic distributions 
are close to their corresponding theoretical repeated sampling distributions, and the asymptotic approximations become better as the sample sizes increase.
\begin{figure}[htb]
    \centering
    \begin{subfigure}[b]{0.32\textwidth}
        \centering
        \includegraphics[height=1.8in]{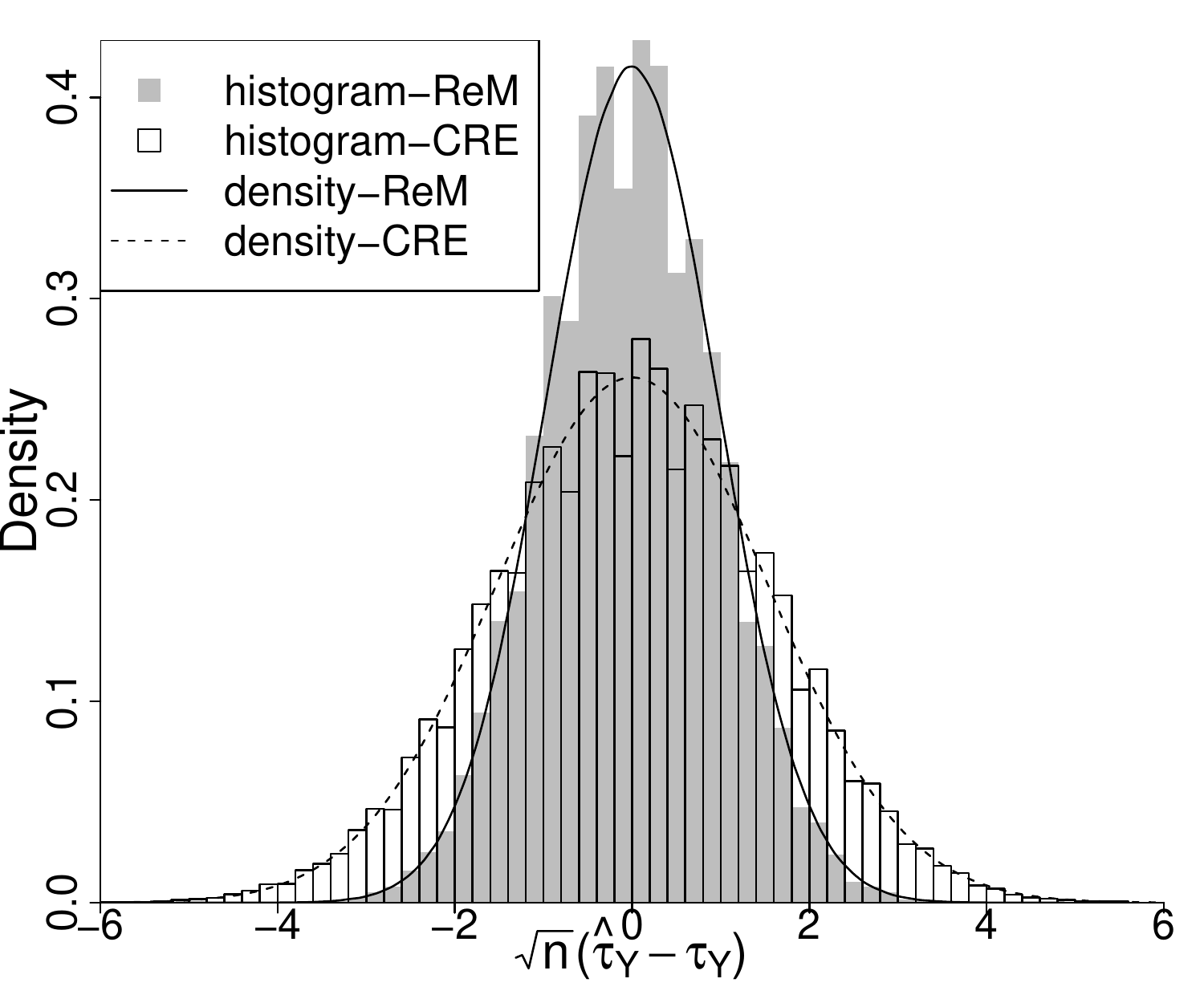}
        \caption{$n=1000$}
    \end{subfigure}%
    \begin{subfigure}[b]{0.32\textwidth}
        \centering
        \includegraphics[height=1.8in]{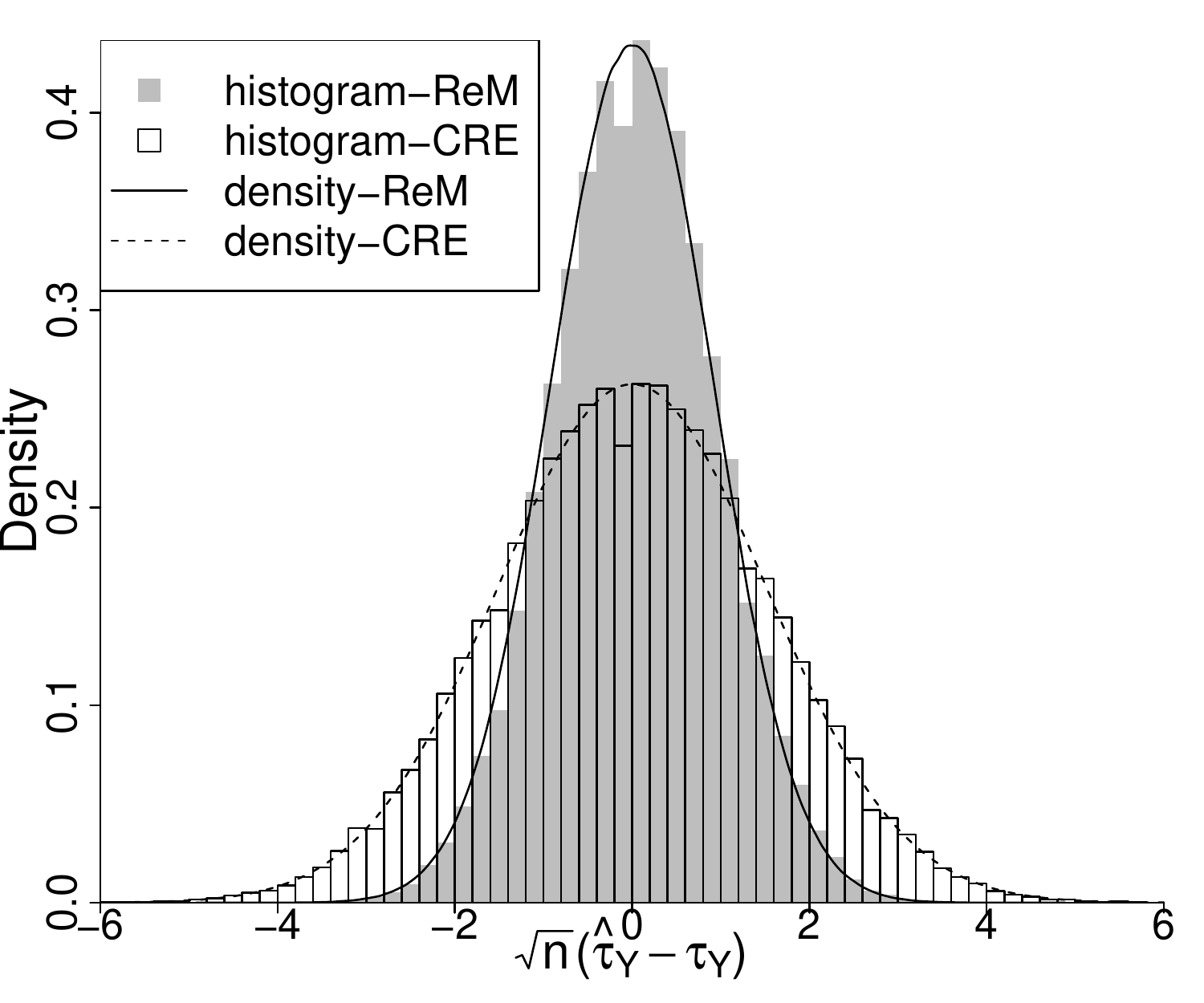}
        \caption{$n=3000$}
    \end{subfigure}%
	\begin{subfigure}[b]{0.32\textwidth}
        \centering
        \includegraphics[height=1.8in]{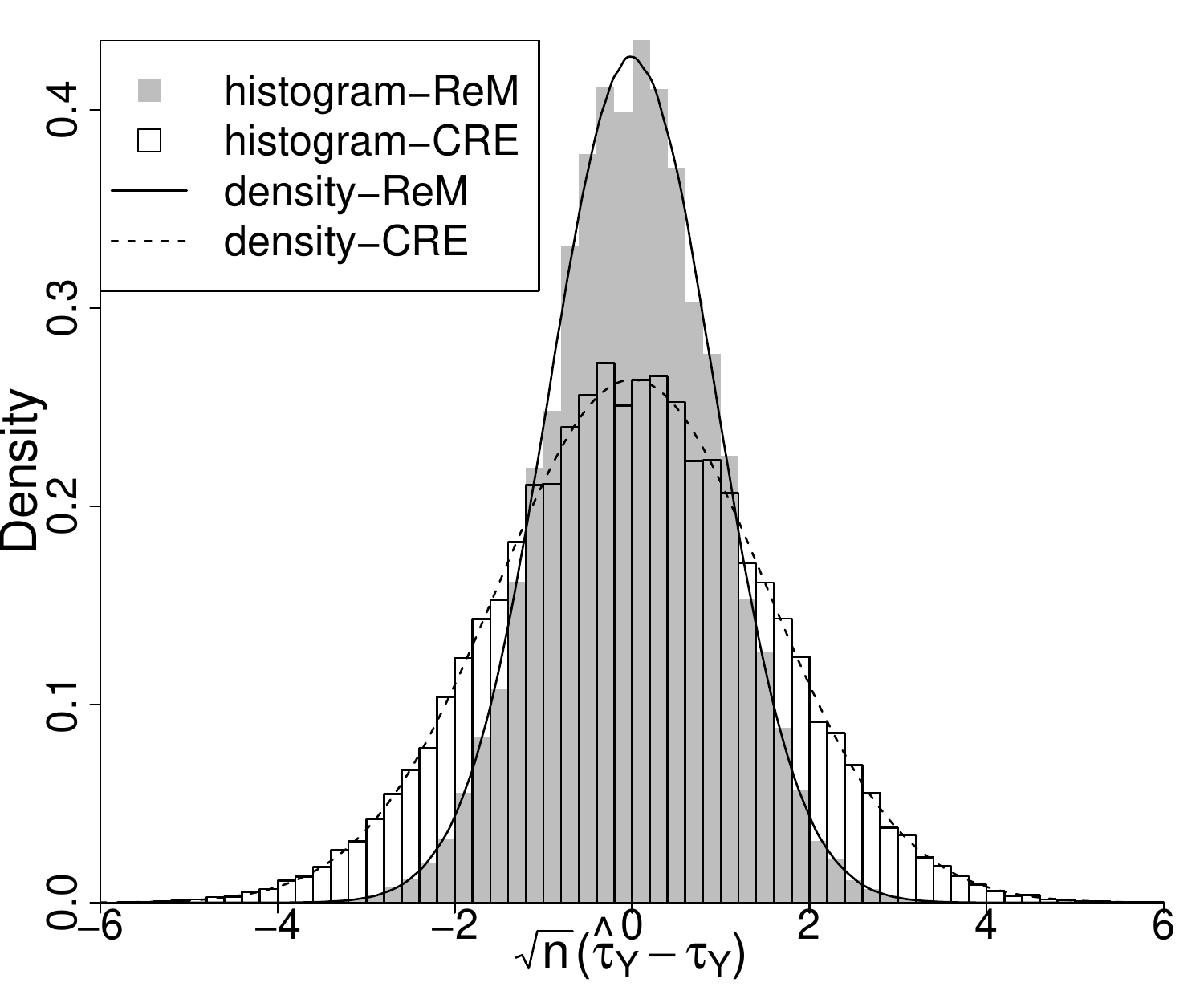}
        \caption{$n=5000$}
    \end{subfigure}
    \caption{Histograms and asymptotic densities of $\sqrt{n}(\hat{\tau}_Y-\tau_Y)$. The grey and white histograms are the empirical distributions, and the solid and dotted lines are the asymptotic densities under ReM and CRE.}\label{fig:binary_sim}
\end{figure}

In the above numerial example, the causal effects are not additive. Figure \ref{fig:binary_sim_coverage} shows the empirical coverage probabilities of the $95\%$ confidence intervals with different sample sizes $(1000, 2000, 3000, 4000, 5000)$ and treatment and control proportions $(r_1, r_0)=(0.1,0.9)$. We also generate other data sets with additive causal effects, in which the data generating process is the same as \eqref{eq:num_gen} except that $Y_i(0)$ is replaced by $Y_i(1)-\tau_Y$. As anticipated, with non-additive causal effects, the empirical coverage probabilities are larger than $95\%$, but with additive causal effects, the empirical coverage probabilities are close to $95\%$.
\begin{figure}[htb]
  \centering
    \includegraphics[width=0.6\textwidth]{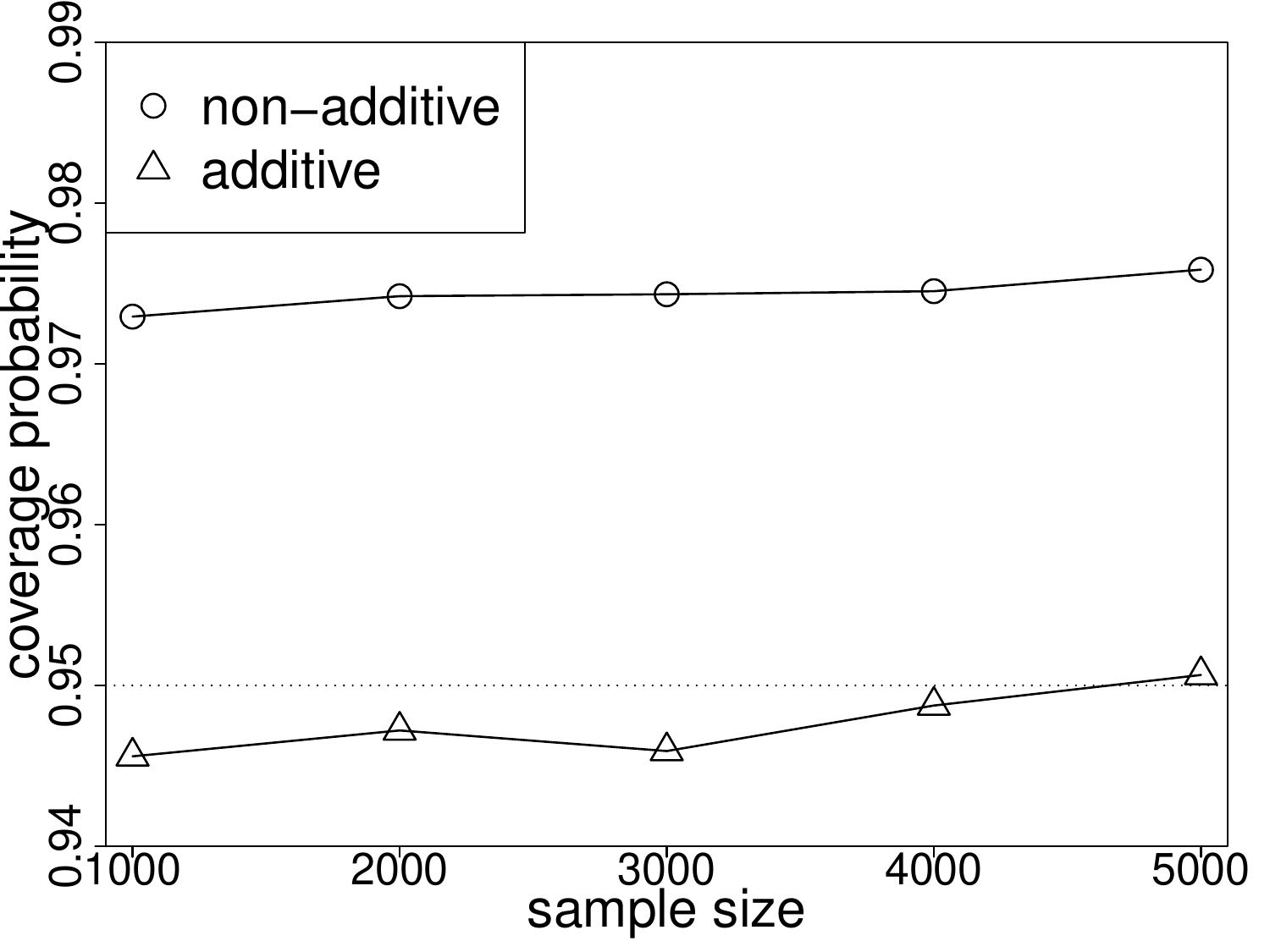}
      \caption{Empirical coverage probabilities of $95\%$ confidence intervals with ($\triangle$) and without ($\bigcirc$) additivity}\label{fig:binary_sim_coverage}
\end{figure}

\section{More Details on Weak Convergence}\label{sec:weak_conv_proof}
We consider the asymptotic distribution of $\sqrt{n}(\hat{\tau}_Y-\tau_Y, \hat{\bm{\tau}}_{\bm{X}}')$ under rerandomization with general balance criterion $\phi(\sqrt{n}\hat{\bm{\tau}}_{\bm{X}}, \bm{V}_{\bm{xx}})$ satisfying the conditions in Section \ref{sec:gen_bal_cri}:
\begin{itemize}
\item[(1)] $\phi$ is almost surely continuous.
\item[(2)] $\Var(\bm{B}\mid \phi(\bm{B},\bm{V}_{\bm{xx}})=1)$, as a function of $\bm{V}_{\bm{xx}}$ with $\bm{B} \sim \mathcal{N}(\bm{0},\bm{V}_{\bm{xx}})$, is continuous for all $\bm{V}_{\bm{xx}}>0$.
\item[(3)] $P(\phi(\bm{B},\bm{V}_{\bm{xx}})=1)>0$, for any  $\bm{V}_{\bm{xx}}>0$ with $\bm{B} \sim \mathcal{N}(\bm{0},\bm{V}_{\bm{xx}})$.
\item[(4)]$\phi(\bm{\mu},\bm{V}_{\bm{xx}})=\phi(-\bm{\mu},\bm{V}_{\bm{xx}})$, for all $\bm{\mu}$ and $\bm{V}_{\bm{xx}}>0$. 
\end{itemize}

We write the limit of $\bm{V}$ as
\begin{align*}
\bm{V}_\infty
=
\lim_{n\rightarrow\infty} \bm{V}
= 
\lim_{n\rightarrow\infty}
\begin{pmatrix}
V_{\tau\tau} & \bm{V}_{\tau \bm{x}}\\
\bm{V}_{\bm{x}\tau} & \bm{V}_{\bm{xx}}
\end{pmatrix} 
= 
\begin{pmatrix}
V_{\tau\tau,\infty} & \bm{V}_{\tau \bm{x},\infty}\\
\bm{V}_{\bm{x}\tau,\infty} & \bm{V}_{\bm{xx},\infty}
\end{pmatrix},
\end{align*}
which is assumed to be positive definite. Let $\mathcal{G} = \{\bm{\mu}: \phi(\bm{\mu}, \bm{V}_{\bm{xx}})=1\}$ be the acceptance region for $\sqrt{n}\hat{\bm{\tau}}_{\bm{X}}$, and $\mathcal{G}_{\infty}=\{\bm{\mu}: \phi(\bm{\mu}, \bm{V}_{\bm{xx},\infty})=1\}$ be its limit.
\begin{proposition}\label{pro:asymp_dist_gen}
Under ReG, as $n\rightarrow \infty$, 
\begin{align*}
\left.
\begin{pmatrix}
\sqrt{n}(\hat{\tau}_Y-\tau_Y) \\
\sqrt{n}\hat{\bm{\tau}}_{\bm{X}}
\end{pmatrix}\ 
\right| \ \sqrt{n}\hat{\bm{\tau}}_{\bm{X}} \in \mathcal{G} \ 
\overset{d}{\longrightarrow} \ 
\left.
\begin{pmatrix}
A_{\infty}\\
\bm{B}_{\infty}
\end{pmatrix}\ 
\right|\ 
 \bm{B}_{\infty} \in \mathcal{G}_{\infty},
\end{align*}
where $(A_{\infty}, \bm{B}_{\infty}')\sim \mathcal{N}(0,\bm{V}_{\infty})$,
in the sense that, for any continuity set $\mathcal{C}$ of $(A_{\infty},\bm{B}'_{\infty})\mid  \bm{B}_{\infty} \in \mathcal{G}_{\infty}$, 
\begin{eqnarray*}
P\left\{
\sqrt{n}(\hat{\tau}_Y-\tau_Y, \hat{\bm{\tau}}_{\bm{X}}') \in \mathcal{C} \mid \sqrt{n}\hat{\bm{\tau}}_{\bm{X}} \in \mathcal{G}
\right\}
\rightarrow
 P\left\{
(A_{\infty}, \bm{B}_{\infty}') \in \mathcal{C}\mid  \bm{B}_{\infty} \in \mathcal{G}_{\infty}
\right\}.
\end{eqnarray*}
\end{proposition}

\begin{proof}[Proof of Proposition \ref{pro:asymp_dist_gen}]
According to the finite population central limit theorem,
\begin{align*}
\left(
\sqrt{n}(\hat{\tau}_Y-\tau_Y), 
\sqrt{n}\hat{\bm{\tau}}_{\bm{X}}', 
\bm{V}_{\bm{xx}}
\right) \ 
\overset{d}{\longrightarrow} \ 
\left(
A_{\infty}, 
\bm{B}_{\infty}', 
\bm{V}_{\bm{xx},\infty}
\right).
\end{align*}
The continuous mapping theorem implies
\begin{align}\label{eq:joint_A_B_phi}
\left(
\sqrt{n}(\hat{\tau}_Y-\tau_Y),
\sqrt{n}\hat{\bm{\tau}}_{\bm{X}}',
\phi( \sqrt{n}\hat{\bm{\tau}}_{\bm{X}}, 
\bm{V}_{\bm{xx}} )
\right) \ 
\overset{d}{\longrightarrow} \ 
\left(
A_{\infty},
\bm{B}_{\infty}',
\phi(\bm{B}_{\infty},\bm{V}_{\bm{xx}, \infty})
\right).
\end{align}
Let $\mathcal{I}=(0.5, 1.5)\in \mathbb{R}^1$ be an open interval. Becasue $\phi$ is a 0-1 function, $\mathcal{I}$ is a continuity set of $\phi(\bm{B}_{\infty},\bm{V}_{\bm{xx},\infty})$, in the sense that $P(\phi(\bm{B}_{\infty},\bm{V}_{\bm{xx},\infty})\in \partial \mathcal{I})=0$.
 According to  (\ref{eq:joint_A_B_phi}) and Portmanteau's Theorem, as $n\rightarrow \infty,$
\begin{align*}
P\left\{
 \phi( \sqrt{n}\hat{\bm{\tau}}_{\bm{X}}, \bm{V}_{\bm{xx}} ) = 1
\right\}
=
P\left\{
 \phi( \sqrt{n}\hat{\bm{\tau}}_{\bm{X}}, \bm{V}_{\bm{xx}} ) \in \mathcal{I}
\right\} \rightarrow
P\left\{
 \phi(\bm{B}_{\infty}, \bm{V}_{\bm{xx},\infty} )\in \mathcal{I}
\right\}
= P\left\{
 \phi(\bm{B}_{\infty}, \bm{V}_{\bm{xx},\infty} ) =1
\right\}.
\end{align*}
For any continuity set $\mathcal{C}\in \mathbb{R}^{K+1}$ of $(A_{\infty},\bm{B}_{\infty}')\mid \phi(\bm{B}_{\infty}, \bm{V}_{\bm{xx},\infty} )=1$, $\mathcal{C}\times\mathcal{I}$ is also a continuity set of $(A_{\infty},\bm{B}_{\infty}', \phi(\bm{B}_{\infty}, \bm{V}_{\bm{xx},\infty} ))$. 
This is because
\begin{eqnarray*}
&&P\left\{
(A_{\infty},\bm{B}_{\infty}', \phi(\bm{B}_{\infty}, \bm{V}_{\bm{xx},\infty} )) \in \partial (\mathcal{C}\times\mathcal{I})
\right\} \\
& \leq & 
P\left\{
(A_{\infty},\bm{B}_{\infty}', \phi(\bm{B}_{\infty}, \bm{V}_{\bm{xx},\infty} )) \in \partial\mathcal{C}\times\mathcal{I}
\right\} +
P\left\{
(A_{\infty},\bm{B}_{\infty}', \phi(\bm{B}_{\infty}, \bm{V}_{\bm{xx},\infty} )) \in \mathcal{C}\times\partial\mathcal{I}
\right\} \\
& & + P\left\{
(A_{\infty},\bm{B}_{\infty}', \phi(\bm{B}_{\infty}, \bm{V}_{\bm{xx},\infty} )) \in \partial \mathcal{C}\times\partial\mathcal{I}
\right\}\\
& = & 
P\left\{
(A_{\infty},\bm{B}_{\infty}') \in \partial \mathcal{C}, \phi(\bm{B}_{\infty}, \bm{V}_{\bm{xx},\infty} ) \in \mathcal{I}
\right\}\\
&=&
P\left\{
(A_{\infty},\bm{B}_{\infty}') \in \partial \mathcal{C}\mid \phi(\bm{B}_{\infty}, \bm{V}_{\bm{xx},\infty} ) \in \mathcal{I}
\right\}\cdot
P\left\{
 \phi(\bm{B}_{\infty}, \bm{V}_{\bm{xx},\infty} ) \in \mathcal{I}
\right\} = 0,
\end{eqnarray*}
where the last equality follows from the fact that $\mathcal{C}$ is a continuity set of $(A_{\infty},\bm{B}_{\infty}')\mid \phi(\bm{B}_{\infty}, \bm{V}_{\bm{xx},\infty} )=1$.
Thus, according to (\ref{eq:joint_A_B_phi}) and Portmanteau's Theorem, as $n\rightarrow \infty,$
\begin{eqnarray*}
&&P\left\{
\sqrt{n}(\hat{\tau}_Y-\tau_Y, \hat{\bm{\tau}}_{\bm{X}}') \in \mathcal{C}, \phi(\sqrt{n}\hat{\bm{\tau}}_{\bm{X}},\bm{V}_{\bm{xx}})=1
\right\} 
 =  P\left\{
\left(
\sqrt{n}(\hat{\tau}_Y-\tau_Y), \sqrt{n}\hat{\bm{\tau}}_{\bm{X}}', \phi(\sqrt{n}\hat{\bm{\tau}}_{\bm{X}},\bm{V}_{\bm{xx}})
\right)
 \in \mathcal{C}\times\mathcal{I}
\right\}
\\ 
 &&\rightarrow 
P\left\{
(A_{\infty}, \bm{B}'_{\infty}, \phi(\bm{B}_{\infty}, \bm{V}_{\bm{xx},\infty} )) \in \mathcal{C}\times\mathcal{I}
\right\}
 =  P\left\{
(A_{\infty}, \bm{B}_{\infty}') \in \mathcal{C}, \phi(\bm{B}_{\infty}, \bm{V}_{\bm{xx},\infty} )=1
\right\}.
\end{eqnarray*}
Hence for any continuity set $\mathcal{C}$ of $(A_{\infty},\bm{B}_{\infty}')\mid \phi(\bm{B}_{\infty}, \bm{V}_{\bm{xx},\infty} )=1$, as $n\rightarrow \infty$,
\begin{eqnarray*}
&&P\left\{
\sqrt{n}(\hat{\tau}_Y-\tau_Y, \hat{\bm{\tau}}_{\bm{X}}') \in \mathcal{C} \mid \phi(\sqrt{n}\hat{\bm{\tau}}_{\bm{X}},\bm{V}_{\bm{xx}})=1
\right\}
=  \frac{P\left\{
\sqrt{n}(\hat{\tau}_Y-\tau_Y, \hat{\bm{\tau}}_{\bm{X}}') \in \mathcal{C}, \phi(\sqrt{n}\hat{\bm{\tau}}_{\bm{X}},\bm{V}_{\bm{xx}})=1
\right\}}{P\left\{
\phi(\sqrt{n}\hat{\bm{\tau}}_{\bm{X}},\bm{V}_{\bm{xx}})=1
\right\}}\\
&&\rightarrow  \frac{P\left\{
(A_{\infty}, \bm{B}_{\infty}') \in \mathcal{C}, \phi(\bm{B}_{\infty}, \bm{V}_{\bm{xx},\infty} )=1
\right\}}{P\left\{
\phi(\bm{B}_{\infty}, \bm{V}_{\bm{xx},\infty} )=1
\right\}}
=  P\left\{
(A_{\infty}, \bm{B}_{\infty}') \in \mathcal{C}\mid \phi(\bm{B}_{\infty}, \bm{V}_{\bm{xx},\infty} )=1
\right\}.
\end{eqnarray*}
Therefore, Proposition \ref{pro:asymp_dist_gen} holds.
\end{proof}
Proposition \ref{pro:asymp_dist_gen} implies the following corollary, including Proposition \ref{thm:asymp_dist_maha} as a special case.
\begin{corollary}\label{cor:asymp_approx_gen}
Under ReG,
\begin{align*}
\left.
\begin{pmatrix}
\sqrt{n}(\hat{\tau}_Y-\tau_Y) \\
\sqrt{n}\hat{\bm{\tau}}_{\bm{X}}
\end{pmatrix}\ 
\right| \ \sqrt{n}\hat{\bm{\tau}}_{\bm{X}} \in \mathcal{G} \ 
\apprsim \ 
\left.
\begin{pmatrix}
A\\
\bm{B}
\end{pmatrix}\ 
\right|\ 
 \bm{B} \in \mathcal{G},
\end{align*}
where $(A, \bm{B}')\sim \mathcal{N}(\bm{0}, \bm{V})$.
\end{corollary}

\begin{proof}[Proof of Corollary \ref{cor:asymp_approx_gen}]
Let $(A_{\infty}, \bm{B}_{\infty}')\sim \mathcal{N}(0,\bm{V}_{\infty})$. As $n\rightarrow\infty$, 
$\bm{V} \rightarrow \bm{V}_{\infty}$, and then
$
\left(
A, \bm{B}', \bm{V}_{\bm{xx}}
\right)
\converged
\left(
A_{\infty}, \bm{B}_{\infty}', \bm{V}_{\bm{xx},\infty}
\right).
$
The same logic as the proof of Proposition \ref{pro:asymp_dist_gen} implies
\begin{align*}
\left.
\begin{pmatrix}
A \\
\bm{B}
\end{pmatrix}\ 
\right| \ \phi(\bm{B}, \bm{V}_{\bm{xx}})=1 \ 
\overset{d}{\longrightarrow} \ 
\left.
\begin{pmatrix}
A_{\infty}\\
\bm{B}_{\infty}
\end{pmatrix}\ 
\right|\ 
\phi(\bm{B}_\infty, \bm{V}_{\bm{xx},\infty})=1
\end{align*}
Therefore, according to Proposition \ref{pro:asymp_dist_gen}, Corollary \ref{cor:asymp_approx_gen} holds.
\end{proof}

The following corollary shows the asymptotic distribution of $\hat{\tau}_Y$ under ReG.
\begin{corollary}\label{cor:asym_tau_Y_ReG}
Under ReG, 
\begin{align*}
\sqrt{n}(\hat{\tau}_Y -\tau_Y) \mid \sqrt{n}\hat{\bm{\tau}}_{\bm{X}} \in\mathcal{G} \apprsim  \varepsilon + \bm{V}_{\tau \bm{x}}\bm{V}_{\bm{xx}}^{-1}\bm{B} \mid \bm{B} \in \mathcal{G},
\end{align*}
where $\varepsilon \sim \mathcal{N}(0, V_{\tau\tau}(1-R^2))$ is independent of $\bm{B}\sim \mathcal{N}(\bm{0},  \bm{V}_{\bm{xx}})$.
\end{corollary}

\begin{proof}[Proof of Corollary \ref{cor:asym_tau_Y_ReG}]
The residual from the linear projection of $A$ on $\bm{B}$ is 
$$
\varepsilon =A - \bm{V}_{\tau \bm{x}}\bm{V}_{\bm{xx}}^{-1}\bm{B} \sim \mathcal{N}(0, (1-R^2)V_{\tau\tau}),
$$ 
which is independent of $\bm{B}$. 
According to Corollary \ref{cor:asymp_approx_gen}, 
\begin{align*}
&\sqrt{n}(\hat{\tau}_Y-\tau_Y) \mid  \sqrt{n}\hat{\bm{\tau}}_{\bm{X}} \in \mathcal{G} \ \  \apprsim \ \  A \mid \bm{B} \in \mathcal{G} \ \ 
 \sim \ \ 
\varepsilon + \bm{V}_{\tau \bm{x}}\bm{V}_{\bm{xx}}^{-1}\bm{B} \mid \bm{B} \in \mathcal{G}.
\end{align*}
\end{proof}

The following corollary shows the asymptotic unbiasedness of $\hat{\tau}_Y$ and balance in means of all covariates, which includes 
Corollaries \ref{thm:unbiased_maha} and \ref{cor:unbiased_tier} as special cases.

\begin{corollary}\label{cor:asym_unbiased_proof}
Under ReG,
$
\mathbb{E}_{\text{a}}\left\{
\sqrt{n}(\hat{\tau}_{{Y}}-\tau_{{Y}}) \mid 
 \sqrt{n}\hat{\bm{\tau}}_{\bm{X}} \in \mathcal{G}
\right\} 
=0.
$
\end{corollary}
\begin{proof}[Proof of Corollary \ref{cor:asym_unbiased_proof}]
According to Proposition \ref{pro:asymp_dist_gen}, $\mathbb{E}_{\text{a}}\left\{
\sqrt{n}(\hat{\tau}_{{Y}}-\tau_{{Y}}) \mid 
 \sqrt{n}\hat{\bm{\tau}}_{\bm{X}} \in \mathcal{G}
\right\} = \mathbb{E}(A_{\infty} \mid \bm{B}_{\infty}\in \mathcal{G}_{\infty})$, where $(A_{\infty}, \bm{B}_{\infty}')\sim \mathcal{N}(0,\bm{V}_{\infty})$. 
Because $\phi$ satisfies $\phi(\bm{\mu}, \bm{V}_{\bm{xx},\infty}) = \phi(-\bm{\mu}, \bm{V}_{\bm{xx},\infty})$,
we know that
 $\bm{B}_{\infty} \in \mathcal{G}_{\infty}$ if and only if $-\bm{B}_{\infty} \in \mathcal{G}_{\infty}$.
Using $(A_{\infty}, \bm{B}_{\infty}')\sim (-A_{\infty}, -\bm{B}_{\infty}')$, we have 
$$
\mathbb{E}(A_{\infty} \mid \bm{B}_{\infty} \in \mathcal{G}_{\infty}) = \mathbb{E}(-A_{\infty}\mid -\bm{B}_{\infty}\in \mathcal{G}_{\infty}) = \mathbb{E}(-A_{\infty}\mid \bm{B}_{\infty}\in \mathcal{G}_{\infty}) = -\mathbb{E}(A_{\infty}\mid \bm{B}_{\infty}\in \mathcal{G}_{\infty}).
$$
Thus, $\mathbb{E}(A_{\infty} \mid \bm{B}_{\infty} \in \mathcal{G}_{\infty})=0$, and $\mathbb{E}_{\text{a}}\left\{
\sqrt{n}(\hat{\tau}_{{Y}}-\tau_{{Y}}) \mid 
 \sqrt{n}\hat{\bm{\tau}}_{\bm{X}} \in \mathcal{G}
\right\} =0$. 
Because covariates are outcomes unaffected by treatment, the difference-in-means of any covariate has asymptotic mean $0$.
\end{proof}

\begin{proof}[{\bf Proof of Proposition \ref{prop:R_sampling_cor}}]
We have
\begin{eqnarray*}
\bm{V}_{\tau \bm{x}}\bm{V}_{\bm{xx}}^{-1}\bm{V}_{\bm{x}\tau}
& = & 
\left(  {r_1^{-1}}\bm{S}_{Y(1),\bm{X}}+{r_0^{-1}}\bm{S}_{Y(0),\bm{X}}\right)\left\{ 
(r_1r_0)^{-1}\bm{S}_{\bm{X}}^2
\right\}^{-1}
\left(
{r_1^{-1}}\bm{S}_{\bm{X}, Y(1)}+{r_0^{-1}}\bm{S}_{\bm{X},Y(0)}
\right)\\
& = & \frac{r_0}{r_1}S_{Y(1)\mid \bm{X}}^2+\frac{r_1}{r_0}S_{Y(0)\mid \bm{X}}^2 +  2\bm{S}_{Y(1), \bm{X}}
\left(
\bm{S}_{\bm{X}}^2
\right)^{-1}\bm{S}_{\bm{X},Y(0)}\\
& = & 
r_{1}^{-1} S_{Y(1)\mid \bm{X}}^2+r_0^{-1}S_{Y(0)\mid \bm{X}}^2 -
\left\{
S_{Y(1)\mid \bm{X}}^2 + S_{Y(0)\mid \bm{X}}^2 -
2\bm{S}_{Y(1), \bm{X}}
\left(
\bm{S}_{\bm{X}}^2
\right)^{-1}\bm{S}_{\bm{X},Y(0)}
\right\}\\
& = & r_{1}^{-1} S_{Y(1)\mid \bm{X}}^2+r_0^{-1}S_{Y(0)\mid \bm{X}}^2 -S_{\tau\mid \bm{X}}^2.
\end{eqnarray*}
The sampling squared multiple correlation between $\hat{\tau}_Y$ and $\hat{\bm{\tau}}_{\bm{X}}$ under CRE has the following equivalent forms: 
\begin{align*}
\text{Corr}(\hat{\tau}_Y, \hat{\bm{\tau}}_{\bm{X}}) & = 
\frac{\bm{V}_{\tau \bm{x}}\bm{V}_{\bm{xx}}^{-1}\bm{V}_{\bm{x}\tau}}{V_{\tau\tau}} = 
\frac{r_1^{-1}S_{Y(1)\mid \bm{X}}^2+r_0^{-1}S_{Y(0)\mid \bm{X}}^2
- S_{\tau\mid \bm{X}}^2
}{{r_1^{-1}}S_{Y(1)}^2 + {r_0^{-1}}S_{Y(0)}^2 - S_{\tau}^2}\\
& = \frac{S^2_{Y(1)}}{r_1 V_{\tau \tau}}R^2(1) + 
\frac{S^2_{Y(0)}}{r_0V_{\tau \tau}}R^2(0) - 
\frac{S^2_{\tau}}{V_{\tau \tau}}R^2(\tau)= R^2.
\end{align*}
Therefore, Proposition \ref{prop:R_sampling_cor} holds.
\end{proof}

\section{Improvements Under Rerandomization}
\label{sec::improvement}

\subsection{Reductions in asymptotic variances}\label{sec:priv_proof}
First we investigate the reduction in asymptotic sampling variances under ReM and ReMT, and then we consider ReG.
We introduce $R_{\infty}^2$ as the limit of $R^2$, and $\rho_{t,{\infty}}^2$ as the limit of $\rho_{t}^2$ ($1\leq t\leq T$). The existences of $R_{\infty}^2$ and $\rho_{t,\infty}^2$ are guaranteed by the convergence of $\bm{V}$.

\begin{proof}[{\bf Proof of Corollary \ref{coro::ReM-variance}}]
Recall that $\bm{B}_{\infty} \sim \mathcal{N}(\bm{0},\bm{V}_{\bm{x}\bm{x},\infty})$.
According to Proposition \ref{pro:asymp_dist_gen} and the results for Gaussian covariates \citep[][Theorem 3.1]{morgan2012rerandomization}, the asymptotic sampling variance of $\sqrt{n}\hat{\bm{\tau}}_{\bm{X}}$ is
\begin{align*}
\Var_{\text{a}}\left(
\sqrt{n}\hat{\bm{\tau}}_{\bm{X}}
\mid \sqrt{n}\hat{\bm{\tau}}_{\bm{X}} \in \mathcal{M}
\right) = \Var(\bm{B}_\infty \mid \bm{B}_{\infty} \in \mathcal{M}_{\infty}) = v_{K,a} \Var(\bm{B}_{\infty}) = v_{K,a} \bm{V}_{\bm{xx},\infty} = \lim_{n\rightarrow\infty}v_{K,a} \bm{V}_{\bm{xx}}.
\end{align*}
Because $\Var_{\text{a}}\left( \sqrt{n}\hat{\bm{\tau}}_{\bm{X}} \right)=\Var(\bm{B}_\infty)=\bm{V}_{\bm{xx},\infty}$, we can deduce
the PRIASV of $\hat{\bm{\tau}}_{\bm{X}}$.

According to Theorem \ref{thm:asymp_dist_Y_maha}, under ReM, the asymptotic sampling variance of $\hat{\tau}_Y$ is
\begin{align*}
&\Var_{\text{a}}\left\{
\sqrt{n}(\hat{\tau}_Y-\tau_Y) \mid  \sqrt{n}\hat{\bm{\tau}}_{\bm{X}} \in \mathcal{G}
\right\}  
=
V_{\tau\tau,\infty}
\left\{
(1-R_{\infty}^2)\Var(\varepsilon_0) + R_{\infty}^2\Var(L_{K,a})
\right\}
= 
V_{\tau\tau,\infty}\left\{
1-R_{\infty}^2 + R_{\infty}^2v_{K,a}
\right\} \\
 =& V_{\tau\tau,\infty}\left\{
1-(1-v_{K,a})R_{\infty}^2
\right\} = \lim_{n\rightarrow\infty}V_{\tau\tau}\left\{
1-(1-v_{K,a})R^2
\right\}.
\end{align*}
Because $\Var_{\text{a}}\left\{
\sqrt{n}(\hat{\tau}_Y-\tau_Y) 
\right\} =V_{\tau\tau,\infty}$, the PRIASV of $\hat{\tau}_Y$ is
\begin{align*}
1 - \frac{\Var_{\text{a}}\left\{
\sqrt{n}(\hat{\tau}_Y-\tau_Y) \mid  \sqrt{n}\hat{\bm{\tau}}_{\bm{X}} \in \mathcal{G}
\right\}}{\Var_{\text{a}}\left\{
\sqrt{n}(\hat{\tau}_Y-\tau_Y) 
\right\} } = (1-v_{K,a})R_{\infty}^2 = \lim_{n\rightarrow\infty}(1-v_{K,a})R^2.
\end{align*}
\end{proof}

\begin{proof}[{\bf Proof of Corollary \ref{coro::ReT-variance}}]
We first derive the asymptotic sampling variance and PRIASV of $\hat{\tau}_Y$, and then derive those of $\hat{\tau}_{\bm{X}}$. 
According to Theorem \ref{thm:asymp_dist_Y_maha_tier}, for ReMT, the asymptotic sampling variance of $\hat{\tau}_Y$ is
\begin{align*}
\Var_{\text{a}}\left\{
\sqrt{n}(\hat{\tau}_Y-\tau_Y) \mid  \sqrt{n}\hat{\bm{\tau}}_{\bm{X}} \in \mathcal{T}
\right\}
 &=  V_{\tau\tau,\infty}\left\{
\rho_{T+1,\infty}^2\Var(\varepsilon_0)+\sum_{t=1}^T \rho_{t,\infty}^2\Var(L_{k_t,a_t})
\right\}\\
& = V_{\tau\tau,\infty}\left\{
\rho_{T+1,\infty}^2 + \sum_{t=1}^T \rho_{t,\infty}^2v_{k_t,a_t}
\right\}\\
&=
V_{\tau\tau,\infty}\left\{
1 - \sum_{t=1}^{T}(1-v_{k_t, a_t})\rho^2_{t,\infty}
\right\} = \lim_{n\rightarrow\infty}V_{\tau\tau}\left\{
1 - \sum_{t=1}^{T}(1-v_{k_t, a_t})\rho^2_{t}
\right\},
\end{align*}
where the last line follows from $\sum_{t=1}^{T+1}\rho_{t,\infty}^2=1$. 
The PRIASV of $\hat{\tau}_Y$ is
\begin{align*}
1 - \frac{\Var_{\text{a}}\left\{
\sqrt{n}(\hat{\tau}_Y-\tau_Y) \mid  \sqrt{n}\hat{\bm{\tau}}_{\bm{X}} \in \mathcal{T}
\right\}}{\Var_{\text{a}}\left\{
\sqrt{n}(\hat{\tau}_Y-\tau_Y) 
\right\}} & = 
\sum_{t=1}^{T}(1-v_{k_t, a_t})\rho^2_{t,\infty} = 
\lim_{n\rightarrow\infty}\sum_{t=1}^{T}(1-v_{k_t, a_t})\rho^2_{t}.
\end{align*}

Let $X[t_j]$ be the $j$th covariate in tier $t$, and $R^2_{\overline{l}, t_j}$ be the finite population squared multiple correlation between ${X}_i[t_j]$ and $\bm{X}_i[\overline{l}]$ for $1\leq l \leq \Tau$, with $R^2_{\overline{0}, t_j}=0$. The PRIASV for the outcome implies that the PRIASV of $\hat{\tau}_{X[t_j]}$ is
\begin{align*}
\lim_{n\rightarrow\infty}\sum_{l=1}^{T}(1-v_{k_l, a_l})\left(
R^2_{\overline{l}, t_j} - R^2_{\overline{l-1}, t_j}
\right) = 
\lim_{n\rightarrow\infty}
\left\{ R^2_{\overline{T}, t_j} - \sum_{l=1}^{T}v_{k_l, a_l}\left(
R^2_{\overline{l}, t_j} - R^2_{\overline{l-1}, t_j}
\right) \right\}.
\end{align*}
Because $R^2_{\overline{l}, t_j}=1$ for $l\geq t$, we can further simplify the PRIASV of $\hat{\tau}_{X[t_j]}$ as
\begin{align*}
\lim_{n\rightarrow\infty}\left\{
1 - \sum_{l=1}^{t}v_{k_l, a_l}\left(
R^2_{\overline{l}, t_j} - R^2_{\overline{l-1}, t_j}
\right) 
\right\}.
\end{align*}
To derive the asymptotic sampling variance of $\hat{\bm{\tau}}_{\bm{X}}$, we use the notation introduced in the proof of Theorem \ref{thm:asymp_dist_Y_maha_tier}.
Let $\bm{\Gamma}_{\infty}$ be the limit of the linear transformation matrix $\bm{\Gamma}$, and $\bm{G}_{\infty}=\bm{\Gamma}_{\infty}\bm{B}_{\infty}=(\bm{G}_{1,\infty}', \bm{G}_{2,\infty}', \ldots, \bm{G}_{T,\infty}')'$ be the block-wise Gram--Schmidt orthogonalization of $\bm{B}_{\infty}$, where $\bm{G}_{t,\infty}$ is a $k_t$ dimensional random vector. 
According to Proposition \ref{pro:asymp_dist_gen} and the fact that $(\bm{G}_{1,\infty}, \bm{G}_{2,\infty}, \ldots, \bm{G}_{T,\infty})$ are mutually independent, the asymptotic sampling variance of $\hat{\bm{\tau}}_{\bm{X}}$ is
\begin{align*}
\Var_{\text{a}}\left(
\sqrt{n}\hat{\bm{\tau}}_{\bm{X}}
\mid \sqrt{n}\hat{\bm{\tau}}_{\bm{X}} \in \mathcal{T}
\right) & = \Var\left(
\bm{B}_{\infty}
\mid \bm{B}_{\infty}\in \mathcal{T}_{\infty}
\right) = \Var\left(
\bm{\Gamma}_{\infty}^{-1}\bm{G}_{\infty}
\mid \bm{G}_{t,\infty}'\Var(\bm{G}_{t,\infty})^{-1}\bm{G}_{t,\infty}\leq a_t, 1\leq t\leq T
\right)\\
& = \bm{\Gamma}_{\infty}^{-1}
\text{diag}\left\{
v_{k_1,a_1}\Var(\bm{G}_{1,\infty}), \ldots, v_{k_T,a_T}\Var(\bm{G}_{T,\infty})
\right\}
\left(\bm{\Gamma}_{\infty}'\right)^{-1}\\
& = \lim_{n\rightarrow\infty} \bm{\Gamma}^{-1}
\text{diag}\left\{
v_{k_1,a_1}(r_1r_0)^{-1}\bm{S}^2_{\bm{E}[1]}, \ldots, v_{k_T,a_T}(r_1r_0)^{-1}\bm{S}^2_{\bm{E}[T]}
\right\}
\left(\bm{\Gamma}'\right)^{-1}\\
& = \lim_{n\rightarrow\infty} \frac{n^2}{n_1n_0} \bm{\Gamma}^{-1}
\text{diag}\left(
v_{k_1,a_1}\bm{S}^2_{\bm{E}[1]}, \ldots, v_{k_T,a_T}\bm{S}^2_{\bm{E}[T]}
\right)
\left(\bm{\Gamma}'\right)^{-1}.
\end{align*}
\end{proof}

According to Proposition \ref{pro:asymp_dist_gen}, for ReG,
\begin{align*}
\Var_{\text{a}}\left(
\sqrt{n}\hat{\bm{\tau}}_{\bm{X}} \mid  \sqrt{n}\hat{\bm{\tau}}_{\bm{X}} \in \mathcal{T}
\right) & = \Var(\bm{B}_{\infty}\mid \bm{B}_{\infty}\in \mathcal{T})
 \equiv 
\bm{V}_{\bm{xx},\phi,\infty}=\lim_{n\rightarrow\infty}\bm{V}_{\bm{xx},\phi},
\\
\Var_{\text{a}}\left\{
\sqrt{n}(\hat{\tau}_Y-\tau_Y) \mid  \sqrt{n}\hat{\bm{\tau}}_{\bm{X}} \in \mathcal{T}
\right\} & = \Var(A_{\infty} \mid \bm{B}_{\infty} \in \mathcal{T}_{\infty}) \\
& = \Var\left(A_{\infty}-\bm{V}_{\tau \bm{x},\infty}\bm{V}_{\bm{xx},\infty}^{-1}B_{\infty} + \bm{V}_{\tau \bm{x},\infty}\bm{V}_{\bm{xx},\infty}^{-1}\bm{B}_{\infty} \mid \bm{B}_{\infty}\in \mathcal{T}_{\infty}\right)\\
& = \Var(A_{\infty}-\bm{V}_{\tau \bm{x},\infty}\bm{V}_{\bm{xx},\infty}^{-1}\bm{B}_{\infty}) + \Var\left( 
\bm{V}_{\tau \bm{x},\infty}\bm{V}_{\bm{xx},\infty}^{-1}\bm{B}_{\infty} \mid \bm{B}_{\infty}\in \mathcal{T}_{\infty}
\right)\\
& = V_{\tau\tau,\infty}(1-R_{\infty}^2) + \bm{V}_{\tau \bm{x},\infty}\bm{V}_{\bm{xx},\infty}^{-1}\bm{V}_{\bm{xx},\phi,\infty}\bm{V}_{\bm{xx},\infty}^{-1}\bm{V}_{\bm{x}\tau,\infty}\\
& =  \lim_{n\rightarrow \infty} \left\{{V}_{\tau\tau}(1-R^2) + \bm{V}_{\tau \bm{x}}\bm{V}_{\bm{xx}}^{-1} 
\bm{V}_{\bm{xx},\phi}
\bm{V}_{\bm{xx}}^{-1}\bm{V}_{\bm{x}\tau}\right\}. 
\end{align*}
We can then immediately check whether ReM reduces the sampling covariance matrix of the difference-in-means of the covariates.

\subsection{Reductions in quantile ranges in ReM and ReMT}\label{sec:quant_reduc_proof}

To prove Theorem \ref{thm:shorter_ci_maha}, we need the following two lemmas. 
\begin{lemma}\label{lemma:order_rho_mono}
Let $\varepsilon_0, \eta\sim \mathcal{N}(0,1)$ be independent. For any $a> 0$ and $c\geq 0$, 
\begin{align*}
P\left(
\sqrt{1-\rho^2} \cdot\varepsilon_0 + \rho \eta \geq c \mid \eta^2 \leq a
\right)
\end{align*}
is a decreasing function of $\rho$ for $\rho \in [0,1]$. 
\end{lemma}
\begin{proof}[Proof of Lemma \ref{lemma:order_rho_mono}]
For any $a> 0$, let $F(\cdot)$ and  $f(\cdot)$ denote the cumulative distribution and probability density of $\mathcal{N}(0,1)$, and let $G(\cdot)$ and $g(\cdot)$ denote the cumulative distribution and probability density of $\eta \mid \eta^2 \leq a$. We have
\begin{align*}
 P\left(
\sqrt{1-\rho^2}\cdot \varepsilon_0 + \rho \eta \geq c \mid \eta^2 \leq a
\right)
& =  \int_{-\infty}^{\infty}P\left( \eta \geq \frac{c - \sqrt{1-\rho^2} \cdot x}{\rho} \mid \eta^2 \leq a\right)\text{d}F(x)\\
& =  \int_{-\infty}^{\infty}
\left\{
1 - G\left( \frac{c - \sqrt{1-\rho^2} \cdot x}{\rho} \right)
\right\}
\text{d}F(x).
\end{align*}
Taking the partial derivative with respect to $\rho$, we have
\begin{eqnarray*}
& & \frac{\partial}{\partial \rho}
P\left(
\sqrt{1-\rho^2} \cdot \varepsilon_0 + \rho \eta \geq c \mid \eta^2 \leq a
\right)
=  \int_{-\infty}^{\infty}
-g\left( \frac{ \sqrt{1-\rho^2}\cdot x-c}{\rho} \right)\frac{x\frac{1}{\sqrt{1-\rho^2}}-c}{\rho^2}
\text{d}F(x)\\
& = &  \int_{-\infty}^{\infty}
-g\left( t \right)\frac{t+c\rho}{\rho(1-\rho^2)} \text{d}F\left(\frac{t\rho+c}{\sqrt{1-\rho^2}}\right)
 =  \int_{-\infty}^{\infty}
-g\left( t \right)\frac{t+c\rho}{\rho(1-\rho^2)}f\left(\frac{t\rho+c}{\sqrt{1-\rho^2}}\right) \frac{\rho}{\sqrt{1-\rho^2}} \text{d}t\\
& = & - \left(1-\rho^2\right)^{-3/2} \int_{-\infty}^{\infty}
g(t)f\left(\frac{t\rho+c}{\sqrt{1-\rho^2}}\right)(t+c\rho)\text{d}t.
\end{eqnarray*}
The integral part in the above formula is
\begin{eqnarray*}
&& \int_{-\infty}^{\infty}
g(t)f\left(\frac{t\rho+c}{\sqrt{1-\rho^2}}\right)(t+c\rho)\text{d}t 
 =  \frac{1}{P(\eta^2 \leq a)} \int_{-\sqrt{a}}^{\sqrt{a}}f(t)f\left(\frac{t\rho+c}{\sqrt{1-\rho^2}}\right) (t+c\rho) \text{d}t\\
& = &
\frac{1}{2\pi P(\eta^2\leq a)}
 \int_{-\sqrt{a}}^{\sqrt{a}}
\exp\left\{ -\frac{t^2}{2} - \frac{(t\rho+c)^2}{2(1-\rho^2)}\right\}
(t+c\rho)
 \text{d}t
 = 
\frac{e^{-c^2/2}}{2\pi P(\eta^2\leq a)} \int_{-\sqrt{a}}^{\sqrt{a}}
\exp\left\{ -\frac{(t+c\rho)^2}{2(1-\rho^2)} \right\}
(t+c\rho)
 \text{d}t\\
& = &
\frac{e^{-c^2/2}}{2\pi P(\eta^2\leq a)} \int_{-\sqrt{a}+c\rho}^{\sqrt{a}+c\rho}
\exp\left\{ -\frac{u^2}{2(1-\rho^2)} \right\}
u
 \text{d}u\geq 0.
\end{eqnarray*}
Therefore, 
$
\partial 
P\left(
\sqrt{1-\rho^2} \cdot \varepsilon_0 + \rho \eta \geq c \mid \eta^2 \leq a
\right) / \partial \rho 
\leq 0.
$
\end{proof}

\begin{lemma}\label{lemma:order_rho_K}
Let $\varepsilon_0 \sim \mathcal{N}(0,1)$, $L_{K,a}\sim D_1\mid \bm{D}'\bm{D}\leq a$, where $\bm{D}=(D_1,\ldots,D_K)' \sim \mathcal{N}(\bm{0}, \bm{I}_K)$, and $(\varepsilon_0, L_{K,a})$ are mutually independent.
Then, for any $a> 0$ and $c\geq 0$, 
\begin{align*}
P\left(\sqrt{1-\rho^2}\cdot\varepsilon_0 + \rho L_{K,a} \geq c \right)
\end{align*}
is a decreasing function of $\rho$ for $\rho \in [0,1]$.
\end{lemma}
{\colpf
\begin{proof}[Proof of Lemma \ref{lemma:order_rho_K}]
The independence of $\varepsilon_0$ and $\bm{D}$ implies
$$
P\left(\sqrt{1-\rho^2} \cdot\varepsilon_0 + L_{K,a} \geq c \right) 
= 
P\left(\sqrt{1-\rho^2}\cdot \varepsilon_0 + \rho D_1 \geq c \mid \bm{D}'\bm{D} \leq a  \right).
$$
Assume $0\leq \rho_1 \leq \rho_2 \leq 1$, and $(d_2, \ldots, d_K)$ satisfies $\sum_{k=2}^K d_k^2<a$.
Conditioning on $(D_2,\ldots,D_K)=(d_2, \ldots, d_K)$, 
Lemma \ref{lemma:order_rho_mono} implies
\begin{align*}
& P\left(\sqrt{1-\rho_1^2}\cdot \varepsilon_0 + \rho_1 D_1 \geq c \mid D_1^2 \leq a-\sum_{k=2}^K D_k^2, D_2=d_2, \ldots,D_{K}=d_K  \right)\\
\geq & P\left(\sqrt{1-\rho_2^2}\cdot\varepsilon_0 + \rho_2 D_1 \geq c \mid D_1^2 \leq a-\sum_{k=2}^K D_k^2, D_2=d_2, \ldots,D_{K}=d_K  \right).
\end{align*}
Taking expection for both sides, we have
\begin{align*}
P\left(\sqrt{1-\rho_1^2}\cdot \varepsilon_0 + \rho_1 D_1 \geq c \mid \bm{D}'\bm{D} \leq a  \right) \geq
P\left(\sqrt{1-\rho_2^2}\cdot \varepsilon_0 + \rho_2 D_1 \geq c \mid \bm{D}'\bm{D} \leq a  \right).
\end{align*}
Therefore, Lemma \ref{lemma:order_rho_K} holds.
\end{proof}
}

\begin{proof}[{\bf Proof of Theorem \ref{thm:shorter_ci_maha}}]
According to Theorem \ref{thm:asymp_dist_Y_maha}, the lengths of $(1-\alpha)$ quantile ranges of the asymptotic distributions of $\sqrt{n}(\hat{\tau}_Y-\tau_Y)$ under ReM and CRE are $2\nu_{1-\alpha/2}(R_{\infty}^2)\sqrt{V_{\tau\tau,\infty}}$ and $2z_{1-\alpha/2}\sqrt{V_{\tau\tau,\infty}}$, respectively.
According to the definition of $\nu_{1-\alpha/2}(R_{\infty}^2)$ and Lemma \ref{lemma:order_rho_K}, 
we know that
$\nu_{1-\alpha/2}(R_{\infty}^2)$ is a decreasing function of $R_{\infty}^2$. 
\end{proof}
To prove Theorem \ref{thm:qr_reduct_maha_tier}, we need the following four lemmas.
We first define a random variable to be SUM  if it is symmetric and unimodal around zero.

\begin{lemma}\label{lemma:order_sum}
Let $\zeta_0,\zeta_1$ and $\zeta_2$ be three jointly  independent random variables. If 
\begin{itemize}
\item[(1)] $\zeta_0$ is continuous and SUM, or $\zeta_0=0$; 
\item[(2)] $\zeta_1$ and $\zeta_2$ are symmetric around 0;
\item[(3)] $P(\zeta_1\geq c) \leq P(\zeta_2\geq c)$ for any $c > 0$, 
\end{itemize}
then $P(\zeta_0+\zeta_1\geq c) \leq P(\zeta_0+\zeta_2\geq c)$ for any $c > 0$.
\end{lemma}
{\colpf
\begin{proof}[Proof for Lemma \ref{lemma:order_sum}]
Note that when $\zeta_0=0$, Lemma \ref{lemma:order_sum} holds automatically. We consider only the case where $\zeta_0$ is continuous and SUM. 
Let $F_{\zeta_0}(\cdot)$ be the cumulative distribution function of ${\zeta_0}$. For any $c> 0$,
\begin{eqnarray*}
& & P(\zeta_0+\zeta_1 \geq c) =  \int_{-\infty}^{\infty} P(\zeta_1 \geq c-x)\text{d} F_{\zeta_0}(x)\\
& = & \int_{-\infty}^{c} P(\zeta_1 \geq c-x)\text{d}F_{\zeta_0}(x) + \int_{c}^{\infty} P(\zeta_1 \geq c-x)\text{d}F_{\zeta_0}(x)\\
& = & \int_{\infty}^{0} P(\zeta_1 \geq t)\text{d}F_{\zeta_0}(c-t)+ \int_{0}^{\infty} P(\zeta_1 \geq -t)\text{d}F_{\zeta_0}(c+t)\\
& = & \int_{0}^{\infty} P(\zeta_1 \geq t)\text{d}\left\{ -F_{\zeta_0}(c-t) \right\}+ \int_{0}^{\infty} P(\zeta_1 \leq t)\text{d}F_{\zeta_0}(c+t)\\
& = & \int_{0}^{\infty} P(\zeta_1 \geq t)\text{d}\left\{ F_{\zeta_0}(t-c)-1 \right\}+ \int_{0}^{\infty} \left\{ 1-P(\zeta_1 \geq t)
\right\}
\text{d}F_{\zeta_0}(c+t)\\
& = & \int_{0}^{\infty}
\text{d}F_{\zeta_0}(c+t) + \int_{0}^{\infty} P(\zeta_1 \geq t)\text{d}\left\{F_{\zeta_0}(t-c) - F_{\zeta_0}(c+t) \right\}\\
& = & \int_{0}^{\infty}
\text{d}F_{\zeta_0}(c+t) + \int_{0}^{\infty} P(\zeta_1 \geq t)\text{d}\left\{ -P(t-c \leq {\zeta_0}\leq t+c)\right\}.
\end{eqnarray*}
Similarly, 
\begin{align*}
P({\zeta_0}+\zeta_2\geq c) 
= & \int_{0}^{\infty}
\text{d}F_{\zeta_0}(c+t) + \int_{0}^{\infty} P(\zeta_2\geq t)\text{d}\left\{ -P(t-c \leq {\zeta_0}\leq t+c)\right\}.
\end{align*}
Because ${\zeta_0}$ is SUM and continuous, 
$-P(t-c \leq {\zeta_0}\leq t+c)$ is a continuous increasing function of $t$ when $t\geq 0$. Because $P(\zeta_1 \geq t)\leq P(\zeta_2 \geq t)$ for any $t> 0$, we have that for all $c > 0$,
\begin{eqnarray*}
&& P({\zeta_0}+\zeta_1 \geq c) 
=  \int_{0}^{\infty}
\text{d}F_{\zeta_0}(c+t) + \int_{0}^{\infty} P(\zeta_1 \geq t)\text{d}\left\{ -P(t-c \leq {\zeta_0}\leq t+c)\right\}\\
& \leq & \int_{0}^{\infty}
\text{d}F_{\zeta_0}(c+t) + \int_{0}^{\infty} P(\zeta_2 \geq t)\text{d}\left\{ -P(t-c \leq {\zeta_0}\leq t+c)\right\}
 =  P({\zeta_0}+\zeta_2 \geq c).
\end{eqnarray*}
\end{proof}
}

\begin{lemma}\label{lemma:SDA_sum}
[Wintner 1936]
If $\zeta_1$ and $\zeta_2$ are SUM and independent, then $\zeta_1+\zeta_2$ is also SUM.
\end{lemma}

\begin{lemma}\label{lemma:order_rho_mono_tier}
Let $\varepsilon_0, \eta_1,\eta_2,\ldots,\eta_T$ be $(\Tau+1)$ i.i.d $\mathcal{N}(0,1)$. 
Let $\{\rho_t\}_{t=1}^{T+1}$ and $\{\tilde{\rho}_t\}_{t=1}^{T+1}$ be
 two nonnegative constant sequences  satisfying $\sum_{t=1}^{T+1} \rho_t^2=\sum_{t=1}^{T+1} {\tilde{\rho}_t}^2= 1.$ If there exists $1\leq t_0\leq T$ such that
$$
\rho_{t_0} \geq \tilde{\rho}_{t_0}, \quad   \rho_t=\tilde{\rho}_{t} \ (t\neq t_0, 1\leq t\leq T),  \quad \rho_{T+1}\leq \tilde{\rho}_{T+1},
$$
then for any $c\geq 0$ and $a_t>0$ ($1\leq t\leq T$),
\begin{align*}
& P\left(
\rho_{T+1}\varepsilon_0 + \sum_{t=1}^\Tau \rho_t \eta_t \geq c \mid  \eta_t^2 \leq a_t, 1\leq t\leq \Tau
\right)\leq
P\left(
\tilde{\rho}_{T+1}\varepsilon_0 + \sum_{t=1}^\Tau \tilde{\rho}_t \eta_t \geq c \mid  \eta_t^2 \leq a_t, 1\leq t\leq \Tau
\right).
\end{align*}
\end{lemma}
\begin{proof}[Proof of Lemma \ref{lemma:order_rho_mono_tier}]
Without loss of generality, we assume $t_0=1$. Then $\rho_1\geq \tilde{\rho}_1$ and $\rho_1^2+\rho_{T+1}^2=\tilde{\rho}_1^2+\tilde{\rho}_{T+1}^2$.
According to Lemma \ref{lemma:order_rho_mono}, for any $c\geq 0$, 
\begin{align*}
P\left( 
\frac{\rho_{T+1}}{\sqrt{\rho_1^2+\rho_{T+1}^2}}
\varepsilon_0 + \frac{\rho_1}{\sqrt{\rho_1^2+\rho_{T+1}^2}}\eta_1 \geq c\mid \eta_1^2 \leq a_1
\right)\leq P\left( 
\frac{\tilde{\rho}_{T+1}}{\sqrt{\tilde{\rho}_1^2+\tilde{\rho}_{T+1}^2}}
\varepsilon_0 + \frac{\tilde{\rho}_1}{\sqrt{\tilde{\rho}_1^2+\rho_{T+1}^2}}\tilde{\rho}_1 \geq c\mid \eta_1^2 \leq a_1
\right)
\end{align*}
which implies that, for any $c\geq 0$, 
\begin{align*}
& P\left(
\rho_{T+1}\varepsilon_0 +\rho_1 \eta_1 \geq c \mid  \eta_1^2 \leq a_1
\right)\leq P\left(
\tilde{\rho}_{T+1}\varepsilon_0 +\tilde{\rho}_1 \eta_1 \geq c \mid  \eta_1^2 \leq a_1
\right).
\end{align*}
According to Proposition  \ref{prop::curious} and Lemma \ref{lemma:SDA_sum},
$
\sum_{t=2}^\Tau \rho_t 
 \left( \eta_t  \mid  \eta_t^2 \leq a_t\right)
$
is SUM. Lemma \ref{lemma:order_sum} implies that for any $c> 0$,
\begin{align*}
& P\left(
\rho_{T+1}\varepsilon_0 + \sum_{t=1}^\Tau \rho_t \eta_t \geq c \mid  \eta_t^2 \leq a_t, 1\leq t\leq \Tau
\right)\leq
P\left(
\tilde{\rho}_{T+1}\varepsilon_0 + \sum_{t=1}^\Tau \tilde{\rho}_t \eta_t \geq c \mid  \eta_t^2 \leq a_t, 1\leq t\leq \Tau
\right).
\end{align*}
Therefore, Lemma \ref{lemma:order_rho_mono_tier} holds.
\end{proof}

\begin{lemma}\label{lemma:order_rho_mono_tier_mult}
Let $\varepsilon_0 \sim \mathcal{N}(0,1)$, $L_{k_t,a_t}\sim D_{t1}\mid \bm{D}_t'\bm{D}_t\leq a_t$, where $\bm{D}_t=(D_{t1},\ldots,D_{tk_t})\sim \mathcal{N}(\bm{0},\bm{I}_{k_t})$, and $(\varepsilon_0, L_{k_1,a_1}, L_{k_2,a_2}, \ldots, L_{k_T,a_T})$ are mutually independent.
Let $\{\rho_t\}_{t=1}^{T+1}$ and $\{\tilde{\rho}_t\}_{t=1}^{T+1}$ be
two nonnegative constant sequences  satisfying $\sum_{t=1}^{T+1} \rho_t^2=\sum_{t=1}^{T+1} {\tilde{\rho}_t}^2= 1.$ If there exists $1\leq t_0\leq T$ such that
$$
\rho_{t_0} \geq \tilde{\rho}_{t_0}, \quad  \rho_t=\tilde{\rho}_{t} \ (t\neq t_0, 1\leq t\leq T),  \quad 
\rho_{T+1}\leq \tilde{\rho}_{T+1},
$$
then for any $c\geq 0$ and $a_t>0$ ($1\leq t\leq T$),
\begin{align*}
& P\left(
\rho_{T+1}\varepsilon_0 + \sum_{t=1}^\Tau \rho_t L_{k_t,a_t}\geq c
\right)\leq
 P\left(
\tilde{\rho}_{T+1}\varepsilon_0 + \sum_{t=1}^\Tau \tilde{\rho}_t L_{k_t,a_t}\geq c
\right).
\end{align*}
\end{lemma}
{\colpf
\begin{proof}[Proof of Lemma \ref{lemma:order_rho_mono_tier_mult}]
Let $\bm{D}_t=(D_{t1},\ldots, D_{tk_t})'\sim \mathcal{N}(\bm{0},\bm{I}_{k_t})$, $1\leq t\leq T$, and $(\bm{D}_1, \ldots, \bm{D}_T, \varepsilon_0)$ be mutually independent.
Conditioning on $D_{tj}=d_{tj}$ $(1\leq t\leq T,j\geq 2)$, according to Lemma \ref{lemma:order_rho_mono_tier},
\begin{eqnarray*}
& & P\left(
\rho_{T+1}\varepsilon_0 + \sum_{t=1}^\Tau \rho_t D_{t1} \geq c \mid  D_{t1}^2 \leq a_t-\sum_{i=2}^{k_t}D_{ti}^2, D_{tj} = d_{tj}, 1\leq t\leq \Tau, j\geq 2
\right)\\
&\leq & P\left(
\tilde{\rho}_{T+1}\varepsilon_0 + \sum_{t=1}^\Tau \tilde{\rho}_t D_{t1} \geq c \mid   D_{t1}^2 \leq a_t-\sum_{i=2}^{k_t}D_{ti}^2, D_{tj} = d_{tj}, 1\leq t\leq \Tau, j\geq 2
\right).
\end{eqnarray*}
Taking expectations of both sides, we have
\begin{eqnarray*}
& & P\left(
\rho_{T+1}\varepsilon_0 + \sum_{t=1}^\Tau \rho_t D_{t1} \geq c \mid  \bm{D}_t'\bm{D}_t \leq a_t, 1\leq t\leq \Tau
\right)\\
& \leq & P\left(
\tilde{\rho}_{T+1}\varepsilon_0 + \sum_{t=1}^\Tau \tilde{\rho}_t D_{t1} \geq c \mid   \bm{D}_t'\bm{D}_t  \leq a_t, 1\leq t\leq \Tau
\right).
\end{eqnarray*}
Therefore, Lemma \ref{lemma:order_rho_mono_tier_mult} holds.
\end{proof}
}

\begin{proof}[{\bf Proof of Theorem \ref{thm:qr_reduct_maha_tier}}]
According to Theorem \ref{thm:asymp_dist_Y_maha_tier}, the lengths of $(1-\alpha)$ quantile ranges of the asymptotic distributions of $\sqrt{n}(\hat{\tau}_Y-\tau_Y)$ under ReMT and CRE are $2\nu_{1-\alpha/2}(\rho_{1,\infty}^2, \rho_{2,\infty}^2, \ldots, \rho_{\Tau,\infty}^2)\sqrt{V_{\tau\tau,\infty}}$ and $2z_{1-\alpha/2}\sqrt{V_{\tau\tau,\infty}}$, respectively.
According to the definition of $\nu_{1-\alpha/2}(\rho_{1,\infty}^2, \rho_{2,\infty}^2, \ldots, \rho_{\Tau,\infty}^2)$ and Lemma \ref{lemma:order_rho_mono_tier_mult}, $\nu_{1-\alpha/2}(\rho_{1,\infty}^2, \rho_{2,\infty}^2, \ldots, \rho_{\Tau,\infty}^2)$ is a decreasing function of $\rho_{t,\infty}^2$,  for any $1\leq t\leq T$. Therefore, Theorem \ref{thm:qr_reduct_maha_tier} holds.
\end{proof}

\section{Conservativeness in Inference}\label{sec::inference}

\subsection{Conservativeness of the sampling variance estimators}\label{sec:conser_var_proof}

The following lemma, which does not require more moment conditions than Condition \ref{cond:fp}, is useful for obtaining asymptotically conservative estimators for the sampling variances and sampling distributions.

\begin{lemma}\label{lemma:cov_AB}
Under ReG, 
if Condition \ref{cond:fp} holds, then for any $(A_i,B_i)$ equal to $(Y_i(1), Y_i(1)), (Y_i(0), Y_i(0)),$ $(Y_i(1), X_{ki}), (Y_i(0), X_{ki})$ or $(X_{ki}, X_{li})$, we have
\begin{align*}
s_{AB}(z)-S_{AB} = o_p(1), \quad (z=0,1)
\end{align*}
$s_{AB}(z)$ is the sample covariance between the $A_i$'s and $B_i$'s under treatment arm $z$, and $S_{AB}$ is the finite population covariance between the $A_i$'s and $B_i$'s. 
\end{lemma}

\begin{proof}[Proof of Lemma \ref{lemma:cov_AB}]
The key is to bound the variance of $s_{AB}(z)$ under ReG. 
According to the law of total expectation, 
\begin{align*}
\mathbb{E}\left[
\left\{
s_{AB}(z) - S_{AB}
\right\}^2
\right] & = 
P\left(\sqrt{n}\hat{\bm{\tau}}_{\bm{X}} \in \mathcal{G}\right) \cdot
\mathbb{E}\left[
\left\{
s_{AB}(z) - S_{AB}
\right\}^2 \mid \sqrt{n}\hat{\bm{\tau}}_{\bm{X}} \in \mathcal{G}
\right] \\
& \quad \ + 
P\left(\sqrt{n}\hat{\bm{\tau}}_{\bm{X}} \notin \mathcal{G} \right) \cdot
\mathbb{E}\left[
\left\{
s_{AB}(z) - S_{AB}
\right\}^2 \mid \sqrt{n}\hat{\bm{\tau}}_{\bm{X}} \notin \mathcal{G} 
\right]\\
& \geq 
P\left(\sqrt{n}\hat{\bm{\tau}}_{\bm{X}} \in \mathcal{G}\right)  \cdot
\mathbb{E}\left[
\left\{
s_{AB}(z) - S_{AB}
\right\}^2 \mid \sqrt{n}\hat{\bm{\tau}}_{\bm{X}} \in \mathcal{G}
\right].
\end{align*}
Therefore, 
\begin{align}\label{eq:bound_mse_sABq_rerand}
\mathbb{E}\left[
\left\{
s_{AB}(z) - S_{AB}
\right\}^2 \mid  \sqrt{n}\hat{\bm{\tau}}_{\bm{X}} \in \mathcal{G}
\right] & \leq P\left( \sqrt{n}\hat{\bm{\tau}}_{\bm{X}} \in \mathcal{G} \right)^{-1}\mathbb{E}\left[
\left\{
s_{AB}(z) - S_{AB}
\right\}^2
\right] \nonumber \\
& = 
 P\left(  \sqrt{n}\hat{\bm{\tau}}_{\bm{X}} \in \mathcal{G} \right)^{-1} \Var\left\{s_{AB}(z)\right\},
\end{align}
where \eqref{eq:bound_mse_sABq_rerand} follows from the fact that $s_{AB}(z)$ is unbiased for $S_{AB}$ under CRE \citep{cochran1977} 
because, under CRE, units receiving treatment arm $z$ is a simple random sample of size $n_z$. 

Let $\bar{A}^\obs(z) = n_z^{-1}\sum_{i:Z_i=z} A_i$ and $\bar{B}^\obs(z) = n_z^{-1}\sum_{i:Z_i=z} B_i$ be the averages of $A_i$'s and $B_i$'s under treatment arm $z$. We first decompose the sample covariance $s_{AB}(z)$ as 
\begin{align*}
s_{AB}(z) & = \frac{n_z}{n_{z}-1} 
\left\{ \frac{1}{n_z}\sum_{i:Z_i=z} (A_i-\bar{A})(B_i - \bar{B}) - \left( \bar{A}^\obs(z)-\bar{A}\right) \left(\bar{B}^\obs(z)-\bar{B}\right) \right\},
\end{align*}
and then obtain an upper bound of its variance under CRE
\begin{align}\label{eq:bound_sABq}
\Var\{s_{AB}(z)\} & =  
\frac{n_z^2}{(n_{z}-1)^2} \Var\left\{\frac{1}{n_z}\sum_{i:Z_i=z} (A_i-\bar{A})(B_i - \bar{B}) - \left( \bar{A}^\obs(z)-\bar{A}\right) \left(\bar{B}^\obs(z)-\bar{B}\right) \right\}
\nonumber
\\
& \leq 
\frac{2n_z^2}{(n_{z}-1)^2}
\left[ \Var\left\{
\frac{1}{n_{z}}\sum_{i:Z_i=z} (A_i-\bar{A})(B_i - \bar{B})
\right\} + 
\Var\left\{
\left( \bar{A}^\obs(z)-\bar{A}\right) \left(\bar{B}^\obs(z)-\bar{B}\right)
\right\}
\right],
\end{align}
which follows from the Cauchy-Schwarz inequality. 
Below we consider the two terms in \eqref{eq:bound_sABq} separately. 
The first term in \eqref{eq:bound_sABq} is bounded by:  
\begin{eqnarray} 
& & \Var\left\{
\frac{1}{n_{z}}\sum_{i:Z_i=z} (A_i-\bar{A})(B_i - \bar{B})
\right\}  \nonumber  \\
& = &
\left(
\frac{1}{n_z} - \frac{1}{n}
\right)
\frac{1}{n-1} \sum_{i=1}^{n}
\left\{
(A_i-\bar{A})(B_i - \bar{B}) - \frac{1}{n} \sum_{j=1}^{n}(A_j-\bar{A})(B_j - \bar{B})
\right\}^2
\nonumber  \\
& \leq & \frac{1}{n_z}\frac{1}{n-1} \sum_{i=1}^{n}(A_i-\bar{A})^2(B_i - \bar{B})^2  \nonumber \\
& \leq &  \frac{n}{n_z} \cdot \frac{1}{n} \max_{1\leq j\leq n} (A_j-\bar{A})^2 \cdot \frac{1}{n-1}\sum_{i=1}^{n}(B_i - \bar{B})^2.  \label{eq::bound1}
\end{eqnarray} 
Because $(\bar{A}^\obs(z)-\bar{A})^2 \leq \max_{1\leq j\leq n} (A_j-\bar{A})^2$, the second term in \eqref{eq:bound_sABq} is bounded by:  
\begin{eqnarray}
& & \Var\left\{
\left( \bar{A}^\obs(z)-\bar{A}\right) \left(\bar{B}^\obs(z)-\bar{B}\right)
\right\}  \nonumber\\
& \leq &  
\mathbb{E}\left\{
\left( \bar{A}^\obs(z)-\bar{A}\right)^2 \left(\bar{B}^\obs(z)-\bar{B}\right)^2
\right\}  \nonumber  \\
& \leq &  \max_{1\leq j\leq n}(A_j-\bar{A})^2 \cdot
\mathbb{E}\left[ \{\bar{B}^\obs(z)-\bar{B}\}^2
\right]   \nonumber \\
& = &
\max_{1\leq j\leq n}(A_j-\bar{A})^2 \cdot
\Var\left\{ \bar{B}^\obs(z) \right\}   \nonumber   \\
& = &  \max_{1\leq j\leq n}(A_j-\bar{A})^2 \cdot
\left(\frac{1}{n_z}-\frac{1}{n} \right) \frac{1}{n-1}\sum_{i=1}^{n}(B_i - \bar{B})^2   \nonumber  \\
& \leq & \frac{n}{n_z} \cdot \frac{1}{n} \max_{1\leq j\leq n} (A_j-\bar{A})^2 \cdot \frac{1}{n-1}\sum_{i=1}^{n}(B_i - \bar{B})^2. \label{eq::bound2}
\end{eqnarray}
Therefore, according to \eqref{eq:bound_sABq}--\eqref{eq::bound2}, we can bound $\Var\{s_{AB}(z)\}$ by: 
\begin{align*}
\Var\{s_{AB}(z)\} & \leq \frac{4n_z^2}{(n_{z}-1)^2} \cdot \frac{n}{n_z} \cdot \frac{1}{n} \max_{1\leq j\leq n} (A_j-\bar{A})^2 \cdot \frac{1}{n-1}\sum_{i=1}^{n}(B_i - \bar{B})^2, 
\end{align*}
which converges to zero under Condition \ref{cond:fp}. 
Recall that $P( \sqrt{n}\hat{\bm{\tau}}_{\bm{X}} \in \mathcal{G})$ converges to the asymptotic acceptance probability $p_a>0$. Then, according to \eqref{eq:bound_mse_sABq_rerand}, $\mathbb{E}[
\{
s_{AB}(z) - S_{AB}
\}^2 \mid  \sqrt{n}\hat{\bm{\tau}}_{\bm{X}} \in \mathcal{G}
] = o(1)$. 
By the Markov inequality, under ReG and Condition \ref{cond:fp}, 
\begin{align*}
s_{AB}(z) - S_{AB} = O_p\left(\sqrt{\mathbb{E}\left[
\left\{
s_{AB}(z) - S_{AB}
\right\}^2 \mid  \sqrt{n}\hat{\bm{\tau}}_{\bm{X}} \in \mathcal{G}
\right]} \right) = o_p(1). 
\end{align*}
Therefore, Lemma \ref{lemma:cov_AB} holds. 
\end{proof}

\begin{lemma}\label{lemma:unbiased_sample_var}
Under ReG, if Condition \ref{cond:fp} holds, then
\begin{align*}
s_{Y(z)}^2-S_{Y(z)}^2 = o_p(1), \ \ 
\bm{s}_{Y(z),\bm{X}}-\bm{S}_{Y(z),\bm{X}} = o_p(1),\ \ 
\bm{s}_{\bm{X},z}^2-\bm{S}_{\bm{X}}^2 = o_p(1), \quad (z=0,1)
\end{align*}
where $\bm{s}_{\bm{X},z}^2$ is the sample variance of $\bm{X}$ under treatment arm $z$.
\end{lemma}

\begin{proof}[Proof of Lemma \ref{lemma:unbiased_sample_var}]
Lemma \ref{lemma:unbiased_sample_var} follows directly from Lemma \ref{lemma:cov_AB}.
\end{proof}

Let 
$\tilde{V}_{\tau\tau} = {V}_{\tau\tau}+ S_{\tau}^2 - S_{\tau\mid \bm{X}}^2 \geq {V}_{\tau\tau}$
and $\tilde{V}_{\tau\tau,\infty}$ be the limit of $\tilde{V}_{\tau\tau}$. Under ReM, according to Lemma \ref{lemma:unbiased_sample_var}, $\hat{V}_{\tau\tau}$ is asymptotically unbiased for $\tilde{V}_{\tau\tau}$, 
and 
$$\hat{V}_{\tau\tau}\hat{R}^2 = r_1^{-1}s_{Y(1)\mid \bm{X}}^2+r_0^{-1}s_{Y(0)\mid \bm{X}}^2
- \left(\bm{s}_{Y(1),\bm{X}}-\bm{s}_{Y(0),\bm{X}}\right) \left(\bm{S}_{\bm{X}}^2\right)^{-1}\left(\bm{s}_{\bm{X},Y(1)}-\bm{s}_{\bm{X},Y(0)}\right)
$$
is asymptotically unbiased for ${V}_{\tau\tau}R^2$.
Thus, the sampling variance estimator is asymptotically unbiased for 
\begin{align*}
\tilde{V}_{\tau\tau} - (1-v_{K,a})V_{\tau\tau}R^2 \geq {V}_{\tau\tau} - (1-v_{K,a})V_{\tau\tau}R^2 \rightarrow \Var_{\text{a}}\left\{
\sqrt{n}(\hat{\tau}_Y-\tau_Y)
\mid \sqrt{n}\hat{\bm{\tau}}_{\bm{X}} \in \mathcal{M}
\right\}.
\end{align*}
Under ReMT, according to Lemma \ref{lemma:unbiased_sample_var}, $\hat{V}_{\tau\tau}$ is asymptotically unbiased for $\tilde{V}_{\tau\tau}$, and $\hat{V}_{\tau\tau}\hat{\rho}_{{t}}^2$ is asymptotically unbiased for ${V}_{\tau\tau}\rho_{{t}}^2$. Thus, the sampling variance estimator is asymptotically unbiased for 
\begin{align*}
\tilde{V}_{\tau\tau} - \sum_{t=1}^{T}(1-v_{k_t, a_t}){V}_{\tau\tau}{\rho}^2_{t}
 \geq {V}_{\tau\tau} - \sum_{t=1}^{T}(1-v_{k_t, a_t}){V}_{\tau\tau}{\rho}^2_{t} \rightarrow
 {\Var}_{\text{a}}\left\{
\sqrt{n}(\hat{\tau}_Y-\tau_Y)
\mid \sqrt{n}\hat{\bm{\tau}}_{\bm{X}} \in \mathcal{T}
\right\}.
\end{align*}
Under ReG, according to Lemma \ref{lemma:unbiased_sample_var}, $\hat{V}_{\tau\tau}$ is asymptotically unbiased for $\tilde{V}_{\tau\tau}$, 
$\hat{V}_{\tau\tau}\hat{R}^2$ is asymptotically unbiased for ${V}_{\tau\tau}R^2$, 
and $\hat{\bm{V}}_{\tau \bm{x}}$ is asymptotically unbiased for $\bm{V}_{\tau \bm{x}}$. Therefore, 
the sampling variance estimator is asymptotically unbiased for 
\begin{align*}
& \tilde{V}_{\tau\tau} -V_{\tau\tau}R^2 + \bm{V}_{\tau \bm{x}}\bm{V}_{\bm{xx}}^{-1}  \bm{V}_{\bm{xx},\phi} \bm{V}_{\bm{xx}}^{-1}\bm{V}_{\bm{x}\tau}\\
\geq & 
{V}_{\tau\tau}(1-R^2) + \bm{V}_{\tau \bm{x}}\bm{V}_{\bm{xx}}^{-1}  \bm{V}_{\bm{xx},\phi} \bm{V}_{\bm{xx}}^{-1}\bm{V}_{\bm{x}\tau}\rightarrow
 {\Var}_{\text{a}}\left\{
\sqrt{n}(\hat{\tau}_Y-\tau_Y)
\mid \sqrt{n}\hat{\bm{\tau}}_{\bm{X}} \in \mathcal{G}
\right\}.
\end{align*}
Above all, the sampling variance estimators are asymptotically conservative.

\subsection{Conservativeness of confidence interval}\label{sec:conser_ci_proof}
First, we consider ReM. 
According to Lemma \ref{lemma:unbiased_sample_var}, 
\begin{align*}
\sqrt{\hat{V}_{\tau\tau}}\left( 
\sqrt{1-\hat{R}^2}\cdot \varepsilon_0 + \sqrt{\hat{R}^2}\cdot L_{K,a}
\right) \converged
\sqrt{\tilde{V}_{\tau\tau,\infty}-V_{\tau\tau,\infty}{R}_{\infty}^2}\cdot \varepsilon_0 + \sqrt{V_{\tau\tau,\infty}{R}_{\infty}^2} \cdot L_{K,a}.
\end{align*}
Thus $\nu_{1-\alpha/2}(\hat{R}^2)\sqrt{\hat{V}_{\tau\tau}}$ is consistent for the $(1-\alpha/2)$th quantile of the distribution on the right hand side of the above formula, which is larger than or equal to
$\nu_{1-\alpha/2}({R}_{\infty}^2)\sqrt{{V}_{\tau\tau,\infty}}$  due to Proposition \ref{prop::curious} and Lemma \ref{lemma:order_sum}.

Second, we consider ReMT. According to Lemma \ref{lemma:unbiased_sample_var},
\begin{align*} 
\sqrt{\hat{V}_{\tau\tau}}
\left(
\hat{\rho}_{\Tau+1} \varepsilon_0 + 
\sum_{t=1}^{T}
\hat{\rho}_{t}
L_{k_t, a_t}
\right) \converged
\sqrt{\tilde{V}_{\tau\tau,\infty}-V_{\tau\tau,\infty} + V_{\tau\tau,\infty}\rho_{\Tau+1,\infty}^2}\cdot \varepsilon_0 + 
\sum_{t=1}^{T}
\sqrt{V_{\tau\tau,\infty}\rho_{t,\infty}^2}\cdot
L_{k_t, a_t}.
\end{align*}
Thus $\nu_{1-\alpha/2}(\hat{\rho}_{1}^2,\ldots,\hat{\rho}_{T}^2)\sqrt{\hat{V}_{\tau\tau}}$ is consistent for the $(1-\alpha/2)$th quantile of the distribution on the right hand side of above formula, which is larger than or equal to
$\nu_{1-\alpha/2}({\rho}_{1,\infty}^2,\ldots,{\rho}_{T,\infty}^2)\sqrt{{V}_{\tau\tau,\infty}}$ due to Proposition \ref{prop::curious}, and Lemmas \ref{lemma:order_sum} and  \ref{lemma:SDA_sum}.

Finally, we consider ReG, and construct the confidence interval. Let 
$\hat{V}_{\varepsilon} =\hat{V}_{\tau\tau}(1-\hat{R}^2)$
be the variance estimator for $ \varepsilon$ in (\ref{eq:asym_gen_re}). 
Let ${q}_{\xi}(\lambda)$ be the $\xi$th quantile of $\sqrt{\lambda}\varepsilon_0 + \hat{\bm{V}}_{\tau \bm{x}}\bm{V}_{\bm{xx}}^{-1}\bm{B} \mid \bm{B} \in \mathcal{G}$, where $\varepsilon_0 \sim \mathcal{N}(0,1)$ is independent of $\bm{B} \sim \mathcal{N}(\bm{0}, \bm{V}_{\bm{xx}})$.
For any $\xi\geq 0.5$, let $\hat{q}_{\xi} = \max_{0\leq \lambda\leq \hat{V}_{\varepsilon}}{q}_{\xi}(\lambda)$. The final confidence interval for $\tau_Y$ is then 
$
\left[
\hat{\tau}_Y - \hat{q}_{1-\alpha/2}/\sqrt{{n}}, 
\hat{\tau}_Y +\hat{q}_{1-\alpha/2}/\sqrt{{n}}
\right].
$
According to Lemma \ref{lemma:unbiased_sample_var}, 
for any $\lambda\geq 0$,
\begin{align}\label{eq:quanti_weak_gen}
\sqrt{\lambda}\varepsilon_0 + \hat{V}_{\tau \bm{x}}\bm{V}_{\bm{xx}}^{-1}\bm{B} \mid \bm{B} \in \mathcal{G} \converged
\sqrt{\lambda}\varepsilon_0 +  {V}_{\tau \bm{x},\infty}\bm{V}_{\bm{xx},\infty}^{-1}\bm{B}_{\infty} \mid \bm{B}_{\infty} \in \mathcal{G}_{\infty},
\end{align}
where $\bm{B}_{\infty} \sim \mathcal{N}(\bm{0}, \bm{V}_{\bm{xx},\infty})$. 
Let $\omega_{1-\alpha/2}(\lambda)$ be the $(1-\alpha/2)$th quantile of the distribution on the right side of (\ref{eq:quanti_weak_gen}). Then 
${q}_{1-\alpha/2}(\lambda)$ is consistent for $\omega_{1-\alpha/2}(\lambda)$. 
According to Lemma \ref{lemma:unbiased_sample_var}, $\hat{V}_{\epsilon}$ is consistent for $\tilde{V}_{\tau\tau,\infty}-V_{\tau\tau,\infty}R_{\infty}^2\geq {V}_{\tau\tau,\infty}(1-R_{\infty}^2)$.
Under some regularity conditions,
\begin{align*}
\hat{q}_{1-\alpha/2} = \max_{0\leq \lambda\leq \hat{V}_{\varepsilon}}{q}_{1-\alpha/2}(\lambda) 
\overset{p}{\longrightarrow} 
\max_{0\leq \lambda\leq \tilde{V}_{\tau\tau,\infty}-V_{\tau\tau,\infty}R_{\infty}^2}{\omega}_{1-\alpha/2}(\lambda) \geq 
\omega_{1-\alpha/2}\left(V_{\tau\tau,\infty}(1-R_{\infty}^2)\right).
\end{align*}
When ${\bm{V}}_{\tau \bm{x}}\bm{V}_{\bm{xx}}^{-1}\bm{B} \mid \bm{B} \in \mathcal{G}$ is unimodal, 
$\hat{q}_{1-\alpha/2} = {q}_{1-\alpha/2}(\hat{V}_{\varepsilon})$ according to 
Lemma \ref{lemma:order_sum}.

%
%

\end{document}